\documentclass[12pt]{amsart}
\usepackage{amssymb}
\newcommand{\C}{\mathbb{C}}
\newcommand{\g}{\mathfrak{g}}
\newcommand{\U}{\mathcal{U}}
\newcommand{\Walg}{\mathcal{W}}

\newcommand{\J}{\mathcal{J}}
\newcommand{\Prim}{\operatorname{Prim}}
\newcommand{\dcell}{{\bf c}}
\newcommand{\Irr}{\operatorname{Irr}}

\newcommand{\Spr}{\operatorname{Spr}}
\newcommand{\Orb}{\mathbb{O}}
\newcommand{\VA}{\operatorname{V}}
\newcommand{\gr}{\operatorname{gr}}
\newcommand{\Hom}{\operatorname{Hom}}
\newcommand{\h}{\mathfrak{h}}
\newcommand{\Z}{\mathbb{Z}}
\newcommand{\Jalg}{\mathsf{J}}

\newcommand{\F}{\mathbb{F}}
\newcommand{\B}{\mathcal{B}}

\newcommand{\I}{\mathcal{I}}
\newcommand{\HC}{\operatorname{HC}}

\newcommand{\Coh}{\operatorname{Coh}}

\newcommand{\slf}{\mathfrak{sl}}
\newcommand{\Q}{\mathbb{Q}}
\newcommand{\M}{\mathcal{M}}

\newcommand{\LL}{\mathcal{L}}
\newcommand{\A}{\mathcal{A}}

\newcommand{\RE}{\mathbb{R}}
\newcommand{\Ce}{\mathbb{C}}
\renewcommand{\t}{{\mathfrak{t}} }

\newtheorem{Thm}{Theorem}[section]
\newtheorem{Prop}[Thm]{Proposition}
\newtheorem{Cor}[Thm]{Corollary}
\newtheorem{Lem}[Thm]{Lemma}
\theoremstyle{definition}

\newtheorem{Rem}[Thm]{Remark}

\title{On dimension growth of modular irreducible representations of semisimple Lie algebras}
\author{Roman Bezrukavnikov and Ivan Losev}
\address{R.B: Department of Mathematics, MIT, Cambridge MA 02139 USA }
\email{bezrukav@math.mit.edu}
\address{I.L.: Department
of Mathematics, Northeastern University, Boston MA 02115 USA \&
International  Laboratory of Representation theory and Mathematical Physics, NRU-HSE, Moscow, Russia}
\email{i.loseu@neu.edu}
\begin{document}
\begin{abstract}
In this paper we investigate the growth with respect to $p$ of dimensions of irreducible representations
of a semisimple Lie algebra $\g$ over $\overline{\mathbb{F}}_p$. More precisely, it is known that for $p\gg 0$,
the irreducibles with a regular rational central character $\lambda$ and $p$-character $\chi$ are indexed by
a certain canonical basis in the $K_0$ of the Springer fiber of $\chi$. This basis is independent of $p$. For a
basis element, the dimension of the corresponding module
is a polynomial in $p$. We show that the canonical basis is compatible with the two-sided
cell filtration for a parabolic subgroup in the affine Weyl group defined by $\lambda$.
We also explain how to read the degree of the dimension polynomial from a filtration component of the basis element.
We use these results to establish conjectures of the second author and
Ostrik on a classification of the finite dimensional irreducible representations of W-algebras, as well
as a strengthening of a result by the first author with Anno and Mirkovic on real variations of stabilities
for the derived category of the Springer resolution.
\end{abstract}
\maketitle
\markright{ON MODULAR IRREDUCIBLE REPRESENTATIONS OF SEMISIMPLE LIE ALGEBRAS}

\begin{center}{\it To the memory of Bertram Kostant.}\end{center}
\section{Introduction}
In this paper we study the representation theory of semisimple Lie algebras over algebraically closed
fields  of big positive characteristic. More precisely, let $G$ be a semisimple algebraic group (of
adjoint type) over $\C$ and $\g$ be its Lie algebra. Then $\g$ is defined over $\Z$ so for an algebraically
closed field $\F$ of characteristic $p$ we can define the form $\g_{\F}$ over $\F$. The universal enveloping algebra $U(\g_\F)$ is finite over its center, namely, we have a central algebra embedding $S(\g^{(1)}_{\F})\hookrightarrow U(\g_{\F}), x\mapsto x^p-x^{[p]}$, where the superscript (1) indicates the Frobenius twist and the superscript $[p]$ stands for the restricted $p$-th power map $\g_{\F}^{(1)}\rightarrow \g_\F$. The image is known as the {\it $p$-center}.
In particular, all irreducible representations of $\g_\F$ are finite dimensional. Below we will assume that
$p\gg 0$ (although some statements hold under weaker assumptions).

Let $\h$ denote a Cartan subalgebra of $\g$. We have an identification
$U(\g_{\F})^{G_{\F}}\xrightarrow{\sim}\F[\h^*]^W$ (the Harish-Chandra isomorphism), the
central subalgebra $U(\g_{\F})^{G_{\F}}\subset U(\g_{\F})$ is known as the {\it Harish-Chandra
center}. Fix $\lambda\in \h^*$ and consider the corresponding central reduction $\U_{\lambda,\F}$
of the algebra $U(\g_\F)$. Further, for $\chi\in \g_{\F}^{(1)*}$ we can consider the further central
reduction $\U^\chi_{\lambda,\F}$, this is a finite dimensional algebra. Obviously, every irreducible
representation of $U(\g_{\F})$ factors through exactly one irreducible quotient $U^{\chi}_{\lambda,\F}$
(some of these quotients are zero).

The study of the representation theory of the algebras $\U_{\lambda,\F}^\chi$ can be easily reduced to
the case when the element $\chi$ is nilpotent, see \cite[Theorem 2]{KW}. Here the algebra $\U_{\lambda,\F}^{\chi}$ is nonzero if
and only if $\lambda\in \h_{\F_p}^*$. Let us recall some results of the first author and collaborators
on the representation theory of $\U^\chi_{\lambda,\F}$.

Consider the flag variety $\mathcal{B}$ for $\g$ (over $\C$).
Let $e$ be a nilpotent element in $\g$ in the orbit corresponding to
that of $\chi$ (since $p\gg 0$, there is a natural bijection between the nilpotent orbits in $\g$
and in $\g_{\F}^{(1)}$). Consider the corresponding Springer fiber $\mathcal{B}_e$. In \cite{BMR}, for a regular $\lambda$, the authors have constructed
 identifications
 \begin{equation}\label{identi}
 K_0(\U_{\lambda,\F}^\chi\operatorname{-mod})\xrightarrow{\sim}
K_0(\operatorname{Coh}(\mathcal{B}_e))\xrightarrow{\sim} H_*(\mathcal{B}_e,\C)
\end{equation} (in the present paper all $K_0$-groups will be over $\C$ but, in fact,  the first isomorphism holds over $\Z$).

There is a way to identify classes of simples under  this isomorphism conjectured by Lusztig and proved
in \cite{BM}. The space $K_0(\operatorname{Coh}(\mathcal{B}_e))$ admits a $q$-deformation, the equivariant $K$-theory group
$K_0(\operatorname{Coh}^{\C^\times}(\mathcal{B}_e))$ for a contracting action of $\C^\times$ on
$\mathcal{B}_e$, \cite[Section 6]{equiv_K}. Then, according to \cite{BM},
there is a canonical basis $\mathfrak{B}$
in $K_0(\operatorname{Coh}^{\C^\times}(\mathcal{B}_e))$ such that the classes of simples in
$K_0(\operatorname{Coh}(\mathcal{B}_e))$ are the specializations of the elements of $\mathfrak{B}$
to $q=1$. The only thing that we need to know about $\mathfrak{B}$ is that it is independent of
$p$ (and depends not on $\lambda$ itself but on its {\it p-alcove}, we will not need this).
%and $\lambda$, provided that $\lambda$ lies in the fundamental $p$-alcove
%(more precisely, $\lambda$ is a reduction modulo $p$ of a weight in the fundamental $p$-alcove).
%see Section \ref{SS_loc_char_p} for further details).

%The space $K_0(\operatorname{Coh}^{\C^\times}(\mathcal{B}_e))$ is an explicit module
%over the  Hecke algebra $\mathcal{H}_q(W^{a})$, where we write $W^{a}$ for the affine Weyl group of $\g$,
%$W^a=W\ltimes Q$, where $Q$ denotes the root lattice.
%
%Lusztig defined a sesquilinear form $(\cdot,\cdot)$ and a bar-involution $\bar{\bullet}$ on
%$K_0(\operatorname{Coh}^{\C^\times}(\mathcal{B}_e))$. From this data, following Kashiwara, one can define
%the canonical basis $\mathfrak{B}$ in  $K_0(\operatorname{Coh}^{\C^\times}(\mathcal{B}_e))$, where the elements
%are defined (up to multiplication by $\pm 1$) by the following two conditions:
%\begin{itemize}
%\item $\bar{b}=b, \forall b\in \mathfrak{B}$,
%\item $(b,b')=\delta_{b,b'}+q \C[q]$.
%\end{itemize}
%Note that the basis is independent of both $p$ and $\lambda$, a basic reason for this will be explained
%in Section \ref{SS_loc_char_p}. The main result of \cite{BM} is that the classes of simples in %$K_0(\U_{\lambda,\F}^\chi\operatorname{-mod})$ coincide with the specialization of the elements of $\mathfrak{B}$
%at $q=1$.

A big problem with this canonical basis is that it is very implicit. For example, it is unclear how to compute the dimensions of the irreducible modules. The goal of this paper is to get a more explicit information about the canonical bases elements and about dimensions of the corresponding simple modules. More precisely, we want to understand the dependence of the dimensions on $p$.

First, let us recall that $K_0(\Coh^{\C^\times}(\mathcal{B}_e))$ is a module over the affine Hecke algebra $\mathcal{H}_q(W^a)$.
Here and below we write $W^a$ for the affine Weyl group of $\g$, i.e., $W^a=W\ltimes Q$, where $Q$
is a root lattice.

Now pick a finite localization $R$ of $\Z$ and a dominant regular element $\lambda\in \h^*_R$.
Then for $p\gg 0$, we can reduce  $\lambda$ to an element in $\h^*_{\mathbb{F}_p}$. Further,
pick $b\in \mathfrak{B}$, and let $V_{\lambda,p}(b)$ denote the corresponding simple
in $\U_{\lambda,\F}^\chi\operatorname{-mod}$. Then (for $\lambda$ and $b$ fixed)
$\dim V_{\lambda,p}(b)$ is known to be a polynomial
in $p$ assuming $p$ satisfies some congruence conditions depending on $\lambda$.
Our first goal is to determine the degree of this polynomial.

Note that $\lambda$ determines a proper standard parabolic subgroup $W_{[\lambda]}\subset W^a$.
Namely, we consider the action of   $W^a$ on $\h^*_{\mathbb{Q}}$.
Let $\lambda^\circ$ be the intersection of $W^a\lambda$ with the fundamental alcove. For $W_{[\lambda]}$
we take the standard parabolic subgroup generated by the simple reflections corresponding to the walls
containing $\lambda^\circ$. For example, when $\lambda\in Q$, we have $W_{[\lambda]}=W$
(as a standard parabolic subgroup of $W^a$).

Consider the partition of $W_{[\lambda]}$ into two-sided cells. This partition also determines
a partition of the irreducible $W_{[\lambda]}$-modules (or $\mathcal{H}_q(W_{[\lambda]})$-modules
for generic $q$) into families.
We filter the module $K_0(\operatorname{Coh}^{\C^\times}(\mathcal{B}_e))$ according to
two-sided cells for $W_{[\lambda]}$. Namely, given a two-sided $\dcell$ for $W_{[\lambda]}$,
let $K_0(\operatorname{Coh}^{\C^\times}(\mathcal{B}_e))_{\leqslant \dcell}$ denote the
intersection of $K_0(\operatorname{Coh}^{\C^\times}(\mathcal{B}_e))$ with the sum of
all irreducible $\mathcal{H}_q(W_{[\lambda]})$-submodules in the localized $K_0$
that belong to families indexed by two-sided cells $\dcell'\leqslant \dcell$.

The following is the main result of the paper. Let us recall that from a two-sided cell $\dcell$ in
$W_{[\lambda]}$ we can recover a nilpotent orbit $\Orb_{\dcell}$ in $\g$,  see
Section \ref{SS_Hecke_HC} for more details.

\begin{Thm}\label{Thm:main}
The following are true:
\begin{enumerate}
\item For any regular $\lambda\in \h^*_R$, the basis $\mathfrak{B}$ of $K_0(\operatorname{Coh}^{\C^\times}(\mathcal{B}_e))$
is compatible with the filtration $K_0(\operatorname{Coh}^{\C^\times}(\mathcal{B}_e))_{\leqslant \dcell}$.
\item Let $b\in \mathfrak{B}$ lie in $K_0(\operatorname{Coh}^{\C^\times}(\mathcal{B}_e))_{\leqslant \dcell}$
but not in smaller filtration pieces. Then the degree of the polynomial $\dim V_{\lambda,p}(b)$
in $p$ equals $\dim \Orb_{\dcell}/2$.
\end{enumerate}
\end{Thm}

\begin{Rem}
There is a classical analog of (2) for categories in characteristic $0$ such as category $\mathcal{O}$.
There the result is that the Gelfand-Kirillov dimension of the module corresponding to a canonical basis element
equals $\dim \Orb_{\dcell}/2$. So part (2) means that the degree of the dimension polynomial is the modular
analog of the Gelfand-Kirillov dimension. Heuristically this can justified as follows: a module of Gelfand-Kirillov dimension $d$ has "the same size" as the space of sections of a coherent sheaf on $\g^*$ with support of dimension $d$, while a module in characteristic $p$ whose dimension $D$ is expressed by a polynomial in $p$ of degree $d$ has the same size as the space of sections  of such a coherent sheaf restricted to the Frobenius
 neighborhood of a point, see also Remark \ref{upper_bound_alt}.

We can also reformulate (2) as follows. We will see below that
there is a unique primitive ideal $\J\in \U$ such that the simple corresponding to $b$ is annihilated
by the reduction of $\J$ mod $p$. We will see that $\overline{\Orb}_{\dcell}$ is the associated variety
of $\J$ so that the degree of the dimension polynomial is $\frac{1}{2}\operatorname{GK-}\dim(\U/\J)$.
We expect that an analog of this result holds in a much greater
generality, for example, for quantizations of symplectic singularities.
\end{Rem}

Let us discuss some applications of Theorem \ref{Thm:main}. First, it allows us to prove conjectures of the second author
and Ostrik on the classification of finite dimensional irreducible modules over the finite W-algebra $\Walg$
for $(\g,e)$, see \cite[Section 7.6]{LO}.  This is Theorem \ref{Thm:Walg_classif_regular} in the paper. In particular, this theorem implies that the $K_0$ of the finite dimensional representations of $\Walg$ with central character $\lambda$
coincides with $\bigoplus_{\dcell} K_0(\operatorname{Coh}(\mathcal{B}_e))_{\leqslant\dcell})$,
where the sum is taken over all two-sided cells in $W_{[\lambda]}$ such that $\Orb_{\dcell}=Ge$.
In fact, for such $\dcell$ we have $K_0(\operatorname{Coh}(\mathcal{B}_e))_{<\dcell})=0$.
The first author and Kazhdan plan to use part (1) and the result mentioned in the previous sentence to study restrictions of characters for unipotent irreducible representations of p-adic groups.

Another application that motivated the main result is a strengthened version of the result of \cite{ABM}.
The central point of {\em loc. cit.} is the definition of  {\em real variation of stabilities}, a concept partly inspired by the notion of a Bridgeland stability condition
on a triangulated category, and a theorem asserting that the categories of $\U_{\lambda,\F}^\chi$-modules
give rise to such a structure.
Let us describe it  in more detail. The above
identification $K_0(\U_{\lambda,\F}^\chi\operatorname{-mod})\xrightarrow{\sim}
K_0(\operatorname{Coh}(\mathcal{B}_{e,\F}))$ comes from
an equivalence of triangulated categories $\LL_{\tilde{\lambda}}: D^b(\U_{\lambda,\F}\operatorname{-mod}^\chi)\to D^b(\operatorname{Coh}_{\mathcal{B}_e} ( T^*\mathcal{B}_\F))$,
(the definition of $\LL_{\tilde{\lambda}}$ is recalled below after Lemma \ref{Lem:K_0_indep}).
Here $\tilde \lambda$ is an element of the root lattice such that $\lambda =\tilde \lambda \mod p$;  %$\tilde {\mathcal{N}}=T^*\mathcal{B}$ is the Springer resolution and
 $\operatorname{Coh}_{\mathcal{B}_e}(T^*\mathcal{B}_\F)$ denotes the category of coherent sheaves on $T^*\mathcal{B}_\F$ set-theoretically
 supported on the closed subvariety $\mathcal{B}_e$, %, where both $\mathcal{B}_e$ and  $\tilde{ \mathcal{N}}$ are considered over the field $\F$;
 while $\U_{\lambda,\F}\operatorname{-mod}^\chi$ is the category of modules
 over $\U_{\lambda,\F}$ where the kernel of $\chi$ acts nilpotently.  The image of the abelian category  $\U_{\lambda,\F}\operatorname{-mod}^\chi$
 under the equivalence $\mathcal{L}_{\tilde \lambda}$, i.e. the corresponding $t$-structure on
 $D^b(\operatorname{Coh}_{\mathcal{B}_\chi} ( T^*\mathcal{B}_\F))$ depends only on the
 $p$-alcove of $\tilde \lambda$, not on $\lambda$ itself. Thus we get a collection of $t$-structures on the derived category of coherent sheaves
 indexed by alcoves; although the above construction applies to varieties of large finite characteristic only, the $t$-structures admits a canonical lift to
 $D^b(\operatorname{Coh}_{\mathcal{B}_e}(T^*\mathcal{B}_\Ce))$.

 It turns out to be a part of a real variation of stability conditions; the content of this statement is as follows:
 for two neighboring alcoves sharing a codimension one face the derived equivalence between the corresponding abelian categories
 is a {\em perverse equivalence} governed by a certain polynomial map $Z:{\mathfrak{t}}^*_{\mathbb{R} } \to K_0(\operatorname{Coh}(\mathcal{B}_e))^*$
 called the central charge map.

 A conjecture stated in {\em loc. cit.}  \cite[Remark 6]{ABM}  asserts that a similar property should hold for two alcoves symmetric
relative to a higher codimension face of the affine coroot stratification of ${\mathfrak{t}}^*_{\mathbb{R} }$. In Section \ref{izvrat} we deduce
(a statement essentially equivalent to) that conjecture from Theorem \ref{Thm:main}. Again, we expect  a similar statement to hold for all (or at least for a wide class of)
symplectic  singularities.

%\begin{Rem}\label{Rem:dimensions} Given an implicit nature of the canonical basis from \cite{BM} getting reasonable
%formulas for the dimension polynomials seems to be out of reach. We still hope that the top degree coefficient admits
%a reasonable formula involving the values at 1 of Kazhdan-Lusztig polynomials for $W_{[\lambda]}$. In \cite{W_dim}
%such formulas were found for finite dimensional irreducible representations of W-algebras with integral central
%character $\lambda$ (that gives dimension polynomials for the irreducible $\U_{\lambda,\F}^\chi$-modules in the lowest
%filtration term).
%\end{Rem}

{\bf Acknowledgements}. We would like to thank Pavel Etingof for a kind permission to use his results in this paper,
see Section \ref{SS_Etingof}, and Victor Ostrik for numerous discussions related to the project.
The work of R.B. was partially supported by the NSF under the grant DMS-1601953. The work of I.L.
has been funded by the  Russian Academic Excellence Project '5-100'
and was partially supported by the NSF under the grant DMS-1501558.

\section{Preliminaries}
\subsection{Harish-Chandra bimodules and primitive ideals}\label{SS_HC_prim}
Let us write $\U$ for $U(\g)$ and $\U_\lambda$ for the central reduction of $\U$ at $\lambda\in \h^*$.

Recall that by a Harish-Chandra (shortly, HC) $\U$-bimodule one means a finitely generated $\U$-bimodule
with locally finite adjoint action of $\g$. In this paper we will only consider the
bimodules where the adjoint $\g$-action integrates to an action of $G:=\operatorname{Ad}(\g)$.
Every HC bimodule admits a so called {\it good filtration}, i.e., a $G$-stable filtration
such that the associated graded is finitely generated as a module over $S(\g)$ (since the filtration
is $G$-stable the left and the right actions of $\g$ on the associated graded coincide).

We will write $\HC(\U)$ for the category of HC $\U$-bimodules and $D^b_{HC}(\U\operatorname{-bimod})$
for the full subcategory of $D^b(\U\operatorname{-bimod})$ of all objects with HC homology. We note
that $D^b_{HC}(\U\operatorname{-bimod})$ is closed under taking the derived tensor products.

Inside $\HC(\U)$ we will consider three kinds of subcategories defined by the central character
conditions. Fix $\lambda,\lambda'\in \h^*$. We consider the subcategory
$\,^1_{\lambda}\HC^1_{\lambda'}$ of all HC bimodules with genuine central characters
$\lambda$ on the left and $\lambda'$ on the right. When $\lambda=\lambda'$, we
write $\HC(\U_\lambda)$ for $\,^1_{\lambda}\HC^1_{\lambda}$. Note that $\HC(\U_\lambda)$ is a monoidal category.
We can also consider the larger subcategory $\,^\infty_{\lambda}\HC^\infty_{\lambda'}$, where the
central characters on the left and on the right are generalized. Finally, there is an intermediate
category $\,^\infty_{\lambda}\HC^1_{\lambda'}$.
Note that $\,^1_{\lambda}\HC^1_{\lambda'}$ contains all simple objects in
$\,^\infty_{\lambda}\HC^\infty_{\lambda'}$ and all objects in
$\,^\infty_{\lambda}\HC^\infty_{\lambda'}$ have finite length. Because of this,
the $K_0$'s of these categories are the same.

Now suppose that $\lambda$ is regular. Let $\mu\in W\lambda$ be anti-dominant meaning that
$\langle\alpha^\vee,\mu\rangle\not\in \Z_{\geqslant 0}$ for any positive coroot $\alpha^\vee$.
Fix this $\mu$ (it is not unique unless
$\lambda$ is integral). Consider the block $\mathcal{O}(\mu)$
of the BGG category $\mathcal{O}$ spanned by the simples $L(u\mu)$ (with highest weight
$u\mu-\rho$), where $u$ is in the integral Weyl group $W_{\mu,int}$ of $\mu$. Recall that this group
is generated by all reflections  $s_\alpha$ such that $\langle\alpha^\vee,\mu\rangle\in \Z$.

Then there is the Bernstein-Gelfand
equivalence $\,^\infty_{\lambda}\HC^1_{\lambda}\xrightarrow{\sim} \mathcal{O}(\mu)$ given
by $\mathcal{M}\mapsto \mathcal{M}\otimes_{\U_\lambda}\Delta(w_0\mu)$, where $w_0$ is the longest
element in $W_{\mu,int}$ (so that $\Delta(w_0\mu)$ is projective in $\mathcal{O}(\mu)$). In particular, the simples
in $\HC(\U_\lambda)$ are labelled by $u\in W_{\mu,int}$. Note that there is a natural
isomorphism $W_{[\mu]}\xrightarrow{\sim} W_{\mu,int}$: namely, let $w_1\in W^a$ be the minimal
length element such that (in the notation of the introduction) $\mu=w_1\mu^\circ$. Then an isomorphism
$W_{[\mu]}\xrightarrow{\sim} W_{\mu,int}$ is given by $w\mapsto \operatorname{pr}(w_1^{-1}ww_1)$,
where we write $\operatorname{pr}$ for the projection $W^a\twoheadrightarrow W$. This defines a
bijection $\operatorname{Irr}(\HC(\U_\lambda))\xrightarrow{\sim} W_{[\lambda]}$.
Let us write $\mathcal{M}_w$ for the simple HC $\U_\lambda$-bimodule corresponding to $w\in W_{[\lambda]}$.

%Both $$  of categories
%are closed with respect to taking tensor products (and Tor's) with an obvious
%compatibility conditions on the central characters. In particular,
%$K_0(\,_\lambda\HC_\lambda)$ is an algebra (where the product comes from
%$\bullet \otimes^L_{\U_\lambda}\bullet$).
%
%The simple objects $\,^1_{\lambda}\HC^1_{\lambda}$  are labelled by the elements of
%$W_{[\lambda]}$ via the Bernstein-Gelfand equivalence with a version of the category
%$\mathcal{O}$. Moreover, as an algebra $K_0(\,_\lambda\HC_\lambda)$ is $\C W_{[\lambda]}$,
%the simple reflections correspond to wall-crossing complexes.

Let $\mathcal{M}$ be a HC $\U$-bimodule. By the associated variety, $\VA(\mathcal{M})$,
we mean the support of $\gr\mathcal{M}$ in $\g$, where the associated graded is taken with respect
to any good filtration. We note that $\VA\left(\operatorname{Tor}^{\U}_i(\mathcal{M}_1,\mathcal{M}_2)\right)
\subset \VA(\mathcal{M}_1)\cap \VA(\mathcal{M}_2)$.

Let us fix a nilpotent orbit $\Orb\subset \g$. We can consider the subcategories $\HC_{\partial\Orb}(\U)
\subset \HC_{\overline{\Orb}}(\U)$ of all $\mathcal{M}\in \HC(\U)$ with
$\VA(\mathcal{M})\subset \partial\Orb$ (resp., $\VA(\mathcal{M})\subset \overline{\Orb}$).
These are tensor ideals in $\HC(\U)$. So we can form the quotient category
$\HC_{\Orb}(\U)$ that also carries the tensor product. Let $\HC_{\Orb}^{ss}(\U)$
denote the full subcategory of semisimple objects in $\HC_\Orb(\U)$. One can show, using,
for example, \cite[Corollary 1.3.2]{HC}, that the subcategory $\HC^{ss}_{\Orb}(\U)$ is closed under
taking the tensor products. Moreover, it is a rigid monoidal category. This has the following corollary.

\begin{Cor}\label{Cor:tens_prod_filtration}
For simple HC bimodules $\M,\M'\in \HC_{\overline{\Orb}}(\U)$, there are $M_1\subset M_2\subset \M\otimes_{\U}\M'$
with $M_1,\M\otimes_\U \M'/M_2\subset \HC_{\partial\Orb}(\U)$, while $M_2/M_1$ is the sum of simple
HC bimodules with associated variety $\overline{\Orb}$.
\end{Cor}

%Let us proceed to primitive ideals.
%Let us start by recalling the Duflo theorem: every primitive ideal in $\U$ is the annihilator
%of an irreducible highest weight module. Let us introduce some notation. For $\lambda\in \h^*$,
%we will write $\Delta(\lambda)$ for the Verma with highest weight $\lambda-\rho$.
%Let $L(\lambda)$ be the irreducible quotient of $\Delta(\lambda)$ and $\J(\lambda)$
%be the annihilator of $L(\lambda)$ in $\U$. So $\J(\lambda)$ and $\J(\lambda')$ have the
%same central character if $\lambda,\lambda'$ are conjugate in $W$.

Let us proceed to primitive ideals (=annihilators of irreducible representations).
We write $\operatorname{Prim}(\U_\lambda)$ for the set of primitive ideals in $\U_\lambda$.
By the Duflo theorem, every primitive ideal in $\U_\lambda$ is the annihilator $\J(\lambda')$
of some irreducible module $L(\lambda')$ with $\lambda'\in W\lambda$.
Inside $\operatorname{Prim}(\U_\lambda)$  we can consider the subset $\operatorname{Prim}_{\Orb}(\U_\lambda)$ of all
$\J$ such that $\VA(\U_\lambda/\J)=\Orb$.

Suppose that $\lambda$ is regular. We have a surjection $W_{[\lambda]}\twoheadrightarrow \Prim(\U_\lambda)$ that sends $w\in W_{[\lambda]}$ to the left annihilator of $\mathcal{M}_w$, let us denote it by $\J_w$. We have
$\J_w=\J(w\mu)$ in our previous notation. The right annihilator of $\M_w$
is $\J_{w^{-1}}$.

Now suppose $\lambda_0$ is singular (and dominant). Pick a strictly dominant element $\mu$ in the root lattice
and let $\lambda=\lambda_0+\mu$ so that, in particular, $\lambda$ is regular dominant.
We have $\J(w'\lambda_0)=\J(w\lambda_0)$ provided $\J(w'\lambda)=\J(w\lambda)$,
see \cite[Section 5.4-5.8]{Jantzen}. This gives the embedding $\Prim(\U_{\lambda_0})\hookrightarrow
\Prim(\U_{\lambda})$ whose image consists of the primitive ideals $\J(w\lambda)$, where
$w$ is longest in $wW_{\lambda_0}$. The embedding sends $\Prim_{\Orb}(\U_{\lambda_0})$  to $\Prim_{\Orb}(\U_\lambda)$.

\subsection{Hecke algebras, cells, and HC bimodules}\label{SS_Hecke_HC}
For a Weyl group $W$ we can consider its Hecke algebra $\mathcal{H}_q(W)$ which comes with the distinguished
basis $c_w,w\in W,$ known as the Kazhdan-Lusztig basis (we use the convention, where the elements $c_w$ are sign-positive with respect to the standard basis $T_w$). This basis allows us to define the so called {\it two-sided pre-order} on the basis elements.  Namely, consider the two-sided based (=spanned by basis elements as a $\C[q^{\pm 1}]$-module)
ideal $I_w$. Set $w\leqslant w'$ if $I_w\subset I_{w'}$. The equivalence classes for this pre-order are known as the {\it two-sided cells}. The induced order on the set of two-sided cells will also be denoted by $\leqslant$. Similarly, we can consider left based ideals, and have the pre-order $\leqslant^L$ and the equivalence relation $\sim^L$ on $W$. The equivalences classes are known as the {\it left cells}.

The two-sided cells and left cells naturally define subquotients of $\mathcal{H}_q(W)$ that are bimodules and
left modules, respectively (called two-sided and left cell modules). The two-sided cell modules allow to partition
irreducible representations of $\mathcal{H}_q(W)$ (and of $\C W$ when $W$ is of finite type) into subsets called
families.

Now let us discuss a connection between the Hecke algebras and HC bimodules. Let $W$ be the Weyl group of $\g$. The category $D^b_{HC}(\U_\lambda\operatorname{-bimod})$ is  monoidal with respect to $\bullet\otimes^L_{\U_\lambda}\bullet$. This monoidal structure equips $K_0(\HC(\U_\lambda))$ with an algebra structure.
The resulting algebra is $\C W_{[\lambda]}$. The class $\mathcal{M}_w$ corresponds to the specialization
of $c_{w^{-1}}$ to $q=1$. The simple reflections in $\C W_{[\lambda]}$ correspond
to the so called wall-crossing bimodules in $K_0(\HC(\U_\lambda))$.

Let us recall the definition of these bimodules. For $\lambda\in \h^*$, let us write $\mathcal{D}_{\mathcal{B}}^\lambda$
for the sheaf of $(\lambda-\rho)$-twisted differential operators. Pick $w\in W_{[\lambda]}$ and view it as an element
of $W_{\mu,int}$ as before. Set $\psi=w_0\mu-w^{-1}w_0\mu$, this is an element of the root lattice.
Let $\mathcal{WC}_w$ denote the $\mathcal{D}^{w_0\mu}_{\mathcal{B}}$-$\mathcal{D}^{w_0\mu-\chi}_{\mathcal{B}}$-bimodule
quantizing the line bundle $\mathcal{O}(\psi)$ on $T^*\mathcal{B}$. Then the wall-crossing bimodule
$\mathsf{WC}_{w}$ is the global sections of $\mathcal{WC}_w$.

Moreover, we get a homomorphism $\operatorname{Br}_{W_{[\lambda]}}\rightarrow D^b_{HC}(\U_\lambda\operatorname{-bimod})$
sending the natural generators of the braid groups to the wall-crossing bimodules,
see  \cite[Section L.3]{Milicic} or \cite[Section 2]{BMR_sing} (that treats the positive characteristic case).
%The class of $\M_w$ is the $c_w$ specialized at $q=1$.
%One also has a graded lift of the category
%$D^b_{HC}(\U_\lambda\operatorname{-bimod})$ (that is still monoidal). The $K_0$ of this graded lift is now
%$\mathcal{H}_q(W_{[\lambda]})$ and the classes of (appropriate graded lifts of simples) are the
%elements $c_w$\footnote{Roma, what is a reference for this in the present generality?}.

Let us now discuss the representation theoretic meaning of cells.
We have $\J_{w}\subset \J_{w'}$ if and only if $w\leqslant^L w'$ and hence
$\J_w=\J_{w'}$ if and only if $w\sim^L w'$, this follows from combining
\cite[Lemma 7.4]{cells4} and \cite[Theorem 3.10]{Joseph_variety}. So if
$w,w'$ are in the same two-sided cell, then $\VA(\M_w)=\VA(\M_{w'})$
(and the converse is true for integral $\lambda$).
So to a two-sided cell we can assign a nilpotent orbit in $\g$, let us denote it by
$\Orb_{\dcell}$.
It is easy to see that $\dcell<\dcell'$ implies $\Orb_\dcell\subsetneq\Orb_{\dcell'}$.

Recall that every left cell contains a so called distinguished (a.k.a. Duflo) involution, say $d$. The corresponding simple
HC bimodule $\M_d$ is the socle of $\U_\lambda/\J_d$. Moreover, the quotient
$\VA((\U_\lambda/\J_d)/\M_d)\subset \partial\Orb_\dcell$.

%\subsection{Restriction functor for HC bimodules}

Now let us discuss asymptotic Hecke algebras. To any Weyl group $W$ Lusztig assigned the so
called asymptotic Hecke algebra $\Jalg=\Jalg(W)$ that is a unital associative algebra (say,
over $\C$) together with a distinguished basis $t_w, w\in W$. The unit in $\Jalg$ is the
element $\sum_{d}t_d$, where the sum is taken over all distinguished involutions in $W$.
There is a homomorphism $\C W\rightarrow \Jalg$ that is known to be an isomorphism when $W$ is of finite type.

%Namely, consider the Hecke algebra $\mathcal{H}_q(W)$ over
%$\C[q^{\pm 1}]$. We can identify $\mathcal{H}_q(W)$ and $\Jalg[q^{\pm 1}]$ as $\C[q^{\pm 1}]$-modules
%by sending the Kazhdan-Lusztig basis element $c_w$ to the basis element $t_w\in \Jalg$.
%Then the map $h\mapsto h*\sum_{d}t_d$, where the sum is taken over
%all distinguished involutions in $W$ and $*$ is the product in $\mathcal{H}_q(W)$, is an algebra
%homomorphism. Specializing to $q=1$, we get a required embedding $\C W\hookrightarrow \Jalg(W)$.

Note that we have $t_wt_{w'}=0$ when $w,w'$ lie in two different two-sided cells.
So we get a decomposition $\Jalg=\bigoplus_{\dcell}\Jalg_{\dcell}$, where $\Jalg_{\dcell}$
is the ideal in $\Jalg$ with basis $t_w, w\in \dcell$. Note that (for $W$ of finite type)
the irreducible $W$-modules that belong to a two-sided cell $\dcell$ are precisely the modules
obtained by pullback $\Jalg_{\dcell}$. Moreover, if $\sigma$ is a left cell in $W$ with distinguished
involution $d$, then the left cell module $[\sigma]$ is $\Jalg(W)t_d$.

Now let us give a categorical interpretation of the algebra $\Jalg(W_{[\lambda]})$.
Consider the rigid monoidal category $\bigoplus_{\Orb} \HC^{ss}_{\Orb}(\U_\lambda)$, where
the sum is taken over all nilpotent orbits in $\g$ (some summands may be zero).
Note that $\HC^{ss}_{\Orb}(\U_\lambda)$ splits as $\bigoplus_{\dcell} \HC^{ss}_{\dcell}(\U_\lambda)$,
where the summand $\HC^{ss}_{\dcell}(\U_\lambda)$ is spanned by the $\M_w$'s with $w\in \dcell$.
So our category can be written as $\bigoplus_{\dcell} \HC^{ss}_{\dcell}(\U_\lambda)$, where the sum
is taken over all two-sided cells in $W_{[\lambda]}$. Then, by the work of Joseph,  e.g.,
\cite{Joseph_cyclicity}, see also \cite{BFO}, it is known that $K_0(\HC^{ss}_{\dcell}(\U_\lambda))=
\Jalg_{\dcell}(W_{[\lambda]})$.
%Moreover, the isomorphism
%$\C W_{[\lambda]}\rightarrow \Jalg(W_{[\lambda]})$ can be interpreted as follows. Recall that
%the classes of simples in $K_0(\HC(\U_\lambda))$ form the Kazhdan-Lusztig basis (specialized
%to $q=1$), see Section \ref{SS_Hecke_HC}. We can identify the spaces
%$K_0(\HC(\U_\lambda))$ and $\bigoplus_{\dcell} K_0(\HC_{\dcell}(\U_\lambda)^{ss})$.  Then the
%algebra homomorphism $\C W_{[\lambda]}\xrightarrow{\sim} \Jalg$ is the map $[\mathcal{M}]\mapsto
%\bigoplus_d [\mathcal{M}\otimes^L_{\U_\lambda}\mathcal{M}_d]$.

\subsection{Localization in characteristic $p$}\label{SS_loc_char_p}
Let us explain results of \cite{BMR,BM} related to the localization in characteristic $p\gg 0$.

Pick a regular dominant element $\lambda\in \h^*_{\mathbb{Q}}$. Let $x$ be the least common denominator
of the coefficients of the simple roots in $\lambda$. In what follows we assume that $p+1$ is
divisible by $x$ so that $(p+1)\lambda$ lies in the root lattice.

Let $\F$ be an algebraically closed field of large enough characteristic $p$. Recall that $\mathcal{B}_{\F}$ stands for the flag variety for $G$ over $\F$. Then we have the sheaf $\mathcal{D}^{\lambda}_{\mathcal{B}_\F}$ that is an Azumaya algebra on $T^*\mathcal{B}^{(1)}_\F$. Note that the categories $\Coh(\mathcal{D}^{\lambda}_{\mathcal{B}_\F}),
\Coh(\mathcal{D}^{\lambda'}_{\mathcal{B}_\F})$ are abelian equivalent, say via twist with a
line bundle $\mathcal{O}(\mu)$, where $\mu$ is a weight congruent to $\lambda'-\lambda$ mod $p$.

We have $R\Gamma(\mathcal{D}^\lambda_{\mathcal{B}_\F})=\U_{\lambda,\F}$.
It was shown in \cite[Section 3.2]{BMR} that the derived global section functor
$R\Gamma: D^b(\operatorname{Coh}(\mathcal{D}^{\lambda}_{\mathcal{B}_\F}))\rightarrow D^b(\U_{\lambda,\F}\operatorname{-mod})$ is an equivalence. Further, it was checked in
\cite[Section 5.4]{BMR} that $\mathcal{D}^{\lambda}_{\mathcal{B}_\F}$ splits in the formal neighborhood
$\mathcal{B}^{(1)\wedge}_{\F,\chi}$ of the Springer fiber in $T^*\mathcal{B}^{(1)}_\F$.
%Moreover, $\U_{0,\F}$ is an Azumaya algebra on $\mathcal{N}_{\F}$, see \cite[Section 5.2]{BMR}, and  %$\mathcal{D}^\lambda_{\mathcal{B}_\F}\otimes
%\rho^*(\U_{0,\F})^{opp}$ splits, \cite[Section 5.1]{BMR}.

Pick a splitting bundle $\mathcal{V}_{\chi,\F}$. This gives rise to the abelian equivalence $$\mathcal{V}_{\chi,\F}\otimes\bullet:\operatorname{Coh}_\chi(T^*\mathcal{B}_\F^{(1)})
\xrightarrow{\sim} \operatorname{Coh}_\chi(\mathcal{D}^{\lambda}_{\mathcal{B}_\F})),$$ where the
subscript $\chi$ refers to the subcategory of sheaves set-theoretically supported at the Springer fiber.
So we arrive at the derived equivalence $D^b_\chi(\operatorname{Coh}(T^*\mathcal{B}_\F^{(1)}))\xrightarrow{\sim} D^b_\chi(\U_{\lambda,\F}\operatorname{-mod})$ given by $M\mapsto R\Gamma(\mathcal{V}_{\chi,\F}\otimes M)$.
The following was shown in \cite[Lemma 6.2.5]{BMR}:

\begin{Lem}\label{Lem:K_0_indep} Fix  $\lambda'$ in the root lattice such that
$\lambda ' = \lambda \mod p$.
Then there exists a canonical choice of the splitting bundle $\mathcal{V}_{\chi,\F}$   (recall it is defined up to a twist with a line
bundle)  such that the class $[\mathcal{V}_{\chi,\F}]\in K_0(\mathcal{B}^{(1)}_{\chi,\F})$ is the pull-back of
$[(\operatorname{Fr}_{\mathcal{B}_\F})_*\mathcal{O}\left((p-1)\rho+\lambda'\right)]$.
% for $\lambda'$ in the fundamental
%$p$-alcove conjugate to $(p+1)\lambda$ under the action of $W^a$.
%The class of $\mathcal{V}_{\chi,\F}$
%of $K_0(\operatorname{Coh}(\mathcal{B}^{(1)}_{\F,\chi}))$ is independent of $\chi$ in the sense that
%it is the pull-back of a class in $K_0(\operatorname{Coh}(\mathcal{B}^{(1)}_\F))$ that is independent
%of $\chi$.
\end{Lem}

The resulting equivalence is denoted by $\LL_{\lambda'}$.

Below we always choose $\mathcal{V}_{\chi,\F}$ as in the lemma.

\begin{Prop}\label{Cor:simples_char_p}
The following is true:
\begin{enumerate}
\item The images of the classes of simple $\U_{\lambda,\F}^\chi$-modules in $K_0(\operatorname{Coh}(\mathcal{B}_\chi))$
are independent of $p$ (as long as $p\gg 0$).
\item The dimensions of the simple $\U_{\lambda,\F}^\chi$-modules are polynomials in $p$
provided $p+1$ is divisible by $x$.
%the dominant weight $\lambda$ varies in a fixed coset of the root lattice. XXX
\end{enumerate}
\end{Prop}
\begin{proof}
(1) for general $\lambda$ follows from the case $\lambda=0$ which is (a) in
\cite[Corollary 5.1.8]{BM}. (2) follows easily from (1) and Lemma
\ref{Lem:K_0_indep}, compare to \cite[Section 6.2]{BMR}.
\end{proof}

We now discuss actions of algebras of interest on the above Grothendieck groups.
Recall that $K_0(\operatorname{Coh}^{\C^\times}(\mathcal{B}_e))$ is a module over the affine Hecke
algebra $\mathcal{H}_q(W^a)$. In particular, $K_0(\operatorname{Coh}(\mathcal{B}_e))\cong
H_*(\mathcal{B}_e,\C)$ acquires an action of $W^a$.

As was shown in \cite[Section 2]{BMR_sing}
the latter action is categorified by an action of the affine braid group $B_{aff}$
on $D^b(\U_{\lambda,\F}\operatorname{-mod}^\chi)\cong D^b(\operatorname{Coh}_{\B_\chi}(T^*\B_\F))$,
while the former one is categorified by a compatible action on the derived
category of a {\em graded version} of $\U_{\lambda,\F}\operatorname{-mod}^\chi$,
which is derived equivalent to $\operatorname{Coh}^{{\mathbb{G}}_m}_{\B_\chi}(T^*\B_\F)$
(see \cite[5.3.1, 5.3.2]{BM}).
% (the full subcategory of
%$D^b(\U_{\lambda,\F}\operatorname{-mod})$ of all objects with homology in $\U_{\lambda,\F}\operatorname{-mod}^\chi$,
%(the category of $\U_{\lambda,\F}$-modules with generalized $p$-character $\chi$),
%where the generator $T_i$ acts by the corresponding wall-crossing functor.

For future reference we mention a standard property of this action.
For a simple reflection $\alpha$ we let $\tilde{s}_\alpha$ denote the corresponding
generator of the affine braid group.
\begin{Lem}\label{either_or}
For a simple reflection $\alpha$ and an irreducible module $L\in\U_{\lambda,\F}\operatorname{-mod}^\chi$
the object $\tilde{s_\alpha}(L)$ either lies in the abelian category $\U_{\lambda,\F}\operatorname{-mod}^\chi$
or  is isomorphic to $L[1]$.
\end{Lem}

\begin{proof}
Consider the full embedding $\U_{\lambda,\F}\operatorname{-mod}^\chi\to
\U_{\F}\operatorname{-mod}^\chi_\lambda$, where the target category
consists of all $\g_\F$-modules where the kernel of the central ideal corresponding to
$(\lambda,\chi)$ acts nilpotently.
By \cite[Theorem 1.3.1]{BR}, we have a compatible $B_{aff}$ action on $D^b(\U_{\F}\operatorname{-mod}^\chi_\lambda)$,
and it suffices to check the same statement in $D^b(\U_{\F}\operatorname{-mod}^\chi_\lambda)$.
We have the exact reflection functor $\Xi_\alpha$ acting on
$\U_{\F}\operatorname{-mod}^\chi_\lambda$ and an exact triangle
$$L\mapsto \Xi_\alpha(L)\to \tilde s_\alpha(L).$$
Recall that $\Xi_\alpha=T_{\mu\to \lambda} \circ T_{\lambda\to \mu}$ is a composition
of two biadjoint translation functors for a weight $\mu$ on the $\alpha$-wall.
If $T_{\lambda\to \mu}(L)=0$ then $\Xi_\alpha(L)=0$ and $\tilde s_\alpha(L)\cong L[1]$.
If $T_{\lambda\to \mu}(L)\ne 0$, then the adjunction arrow $L\to T_{\mu\to \lambda} \circ T_{\lambda\to \mu}
(L)$ is nonzero, hence it is injective provided that $L$ is irreducible. Thus in this case
$\tilde s_\alpha(L)\cong \Xi_\alpha(L)/L$ is concentrated in homological degree zero.
\end{proof}

 \begin{Rem}
 It is natural to expect that the aforementioned action of %$B_{aff}$
 the affine braid group on
 the derived categories of coherent sheaves factors through the
 standard categorification of the  affine Hecke algebra; the latter
 can be defined either using constructibe sheaves on the affine
 flag variety, or using the theory of Soergel bimodules.
 For a base field of characteristic zero this follows from the main
 result of \cite{B}, see also \cite{BY} for the relation to Soergel bimodules.
 For a base field of positive characteristic (which is the setting
 related to $\g$-modules in positive characteristic as explained above)
 this question is still open, to the authors' knowledge.
 \end{Rem}

\section{Lengths}
This section contains a number of results that will be used to prove Theorem
\ref{Thm:main}.

\subsection{Reduction of HC bimodules to characteristic $p$}\label{SS_HC_red_mod_p}
The proof of Proposition \ref{Prop:char_p_length} will be  based on considering
reductions of HC bimodules to characteristic $p$.

Let us start by discussing $R$-forms of Harish-Chandra bimodules. The category
of HC bimodules is defined over $\Q$, the Bernstein-Gelfand equivalence
shows that $\,^\infty_{\lambda}\HC^1_{\lambda}$ is split over $\Q$ because the category
$\mathcal{O}$ is split over the rationals. Recall that an abelian category equivalent to
a category of modules over a finite dimensional algebra over a field is called split if the endomorphism
algebras  of all simples coincide with the field.

Clearly, there is a finite localization $R$ of $\Z$ such that the tensor category $D^b_{HC}(\U_\lambda\operatorname{-bimod})$
is defined over $R$. All simples are defined over $R$ as well, let us fix some $R$-lattices $\M_{w,R},
w\in W_{[\lambda]}$. Note that we can still talk about HC $\U_{\lambda,R}$-bimodules: these are bimodules $\M$
that admit a bounded from below good filtration (such that the left and the right actions of
$R[\mathcal{N}]$ on $\gr\M$ coincide and the $R[\mathcal{N}]$-module $\gr\M$ is finitely generated
-- here $\mathcal{N}$ stands for the nilpotent cone of $\g$). In particular, every HC $\U_{\lambda,R}$-
(or $\U_R$-) bimodule becomes flat over $R$ after a finite localization. Note also that  any Tor of any two HC $\U_{\lambda,R}$-bimodules is again HC.

For a primitive ideal $\J\subset \U_\lambda$ we set $\J_R:=\U_{\lambda,R}\cap \J$.

Let $V=V(\mu)$ denote the irreducible
$G$-module with highest weight $\mu$.  For $m\in \Z_{>0}$, we write
$$V^m:=\bigoplus_{\mu| \langle\rho^\vee,\mu\rangle\leqslant m}V(\mu)^{\dim V(\mu)}.$$

We will also impose the following conditions that we can achieve by a finite
localization of $R$ (in (c3),(c4) we fix $m$ and then further localize $R$).
Here (c2) follows from Corollary \ref{Cor:tens_prod_filtration}, that is an analogous statement over $\C$.

\begin{itemize}
\item[(c1)] For every distinguished involution $d$, we have an inclusion $\M_{d,R}\hookrightarrow
\U_{\lambda,R}/\J_{d,R}$ and the quotient is filtered by bimodules $\M_{w,R}$ with $w<^L d$.
\item[(c2)] For every $w_1,w_2\in \dcell$, there are $\U_{\lambda,R}$-subbimodules
$M_1\subset M_2\subset \M_{w_1,R}\otimes_{\U_{R}}\M_{w_2,R}$ such that
\begin{itemize}
\item
both $M_1$ and $\M_{w_1,R}\otimes_{\U_{\lambda,R}}\M_{w_2,R}/M_2$ are filtered by
$\M_{w,R}$'s, where $w$ lie in cells strictly less then $\dcell$
\item and $M_2/M_1$ is isomorphic to the direct sum of $\M_{w,R}$'s for $w\in \dcell$.
\end{itemize}
\item[(c3)] Both $\operatorname{pr}_\lambda(V^m\otimes \M_w)$ and its complement in $V^m\otimes \M_w$
are defined over $R$ for all $w\in W_{[\lambda]}$.
\item[(c4)] $\operatorname{pr}_\lambda(V^m\otimes \M_w)_R$ is filtered by $\M_{w',R}$'s.
Moreover,  there is a quotient of  $\operatorname{pr}_\lambda(V^m\otimes \M_w)_R$
isomorphic to the direct sum of $\M_{w,R}$'s that gives
$\mathsf{head}(\operatorname{pr}_\lambda(V^m\otimes \M_w))$
after base change to $\C$ (recall that by the head we mean the maximal semisimple
quotient).
\item[(c5)] The wall-crossing bimodules are defined over $R$ and define a homomorphism
$\operatorname{Br}_{W_{[\lambda]}}\rightarrow D^b_{HC}(\U_{\lambda,R}\operatorname{-bimod})$.
\item[(c6)] All $\operatorname{Tor}^i_{\U_{\lambda,R}}(\M_{w_1,R},\M_{w_2,R})$
are filtered by $\M_{w,R}$'s (note that after a finite localization of $R$ only finitely
many of these Tor's are nonzero because $\U_\lambda$ has finite homological dimension). The analogous result is true for
$\operatorname{Tor}^i_{\U_{\lambda,R}}(\M_{w_1,R},\U_{\lambda,R}/\J_{w_2,R})$.
\end{itemize}

Now let $\F$  be an algebraically closed field and an  $R$-algebra.  We have an action of $D^b_{HC}(\U_{\lambda,R}\operatorname{-bimod})$ on $D^b(\U_{\lambda,\F}\operatorname{-mod}^\chi)$ that gives rise to an action of $W_{[\lambda]}$ on $K_0(\U^\chi_{\lambda,\F}\operatorname{-mod})$.

%On the other hand, we get a $W_{[\lambda]}$-action on $K_0(\U_{\lambda,\F}\operatorname{-mod}_\chi)$
%coming from the Harish-Chandra bimodules. Namely, we can choose a finite localization $R$ of $\Z$
%such that:
%\begin{enumerate}
%\item all simples in $\HC(\U_\lambda)$ are defined over $R$,
%\item the derived tensor product is defined over $R$.
%\end{enumerate}
%
%Then, for a simple $\mathcal{M}\in \HC(\U_\lambda)$ and $L\in \U_{\lambda,\F}\operatorname{-mod}_{\chi}$
%we can define $[\mathcal{M}][L]$ as the class of $\mathcal{M}_\F\otimes^L_{\U_{\lambda,\F}}L$, where
%we write $\mathcal{M}_\F$ for $\F\otimes_R \mathcal{M}_R$ and $\mathcal{M}_R$ is an $R$-form of $\mathcal{M}$.
%This gives rise to a well-defined $K_0(\HC(\U_\lambda))$-module structure on %$K_0(\U_{\lambda,\F}\operatorname{-mod}_\chi)$.
%

\begin{Lem}\label{Lem:action_compat}
Under the identification $K_0(\operatorname{Coh}(\mathcal{B}_\chi))\cong K_0(\U_{\lambda,\F}\operatorname{-mod}^\chi)$,
the two actions of $W_{[\lambda]}$ (i.e., the one defined above and the one restricted from the
$W^a$-action in Section \ref{SS_loc_char_p}) coincide.
\end{Lem}
\begin{proof}
The $W^a$-action on $K_0(\operatorname{Coh}(\mathcal{B}_\chi))$ corresponds to the action
on $K_0(\U_{\lambda,\F}\operatorname{-mod}^\chi)$ by the wall-crossing functors, \cite[Section 5.4]{Riche}.
By \cite[Theorem 2.1.4]{BMR_sing}, the wall-crossing functors through the walls defined by the simple
roots for $W_{[\lambda]}$ are given by taking the derived tensor products with the wall-crossing bimodules.
\end{proof}

Below we are also going to use the following lemma.

\begin{Lem}\label{Lem:HC_der_prod}
Let $L$ be a simple $\U_{\lambda,\F}^\chi$-module and let $d$ be a distinguished involution
such that $\J_{d}$ is a maximal primitive ideal with $\J_{d,R}L=0$. Then the following is true:
\begin{enumerate}
\item For any $w\in W$, all simple constituents of $\operatorname{Tor}_*^{\U_{\lambda,R}}(\M_{w,R},L)$
are annihilated by $\J_{w,R}$.
\item Let $L'$ be a simple constituent of $\operatorname{Tor}_*^{\U_{\lambda,R}}(\M_{w,R},L)$
and let $\J_{d'}$ be a maximal primitive ideal such that $\J_{d',R}L'=0$. Then the two-sided
cell of $d'$ is less than or equal to that of $d$.
\item Assume, in addition, that $d\not\sim^Lw^{-1}$. Let $L'$ be a simple constituent
of $\operatorname{Tor}_*^{\U_{\lambda,R}}(\M_{w,R},L)$ and $d'\in W$ be as in (2). Then $d'<^L w$.
\end{enumerate}
\end{Lem}
\begin{proof}
Let us take a resolution of $L$ by free $\U_{\lambda,R}$-modules:
$$\ldots\rightarrow\U_{\lambda,R}^{\oplus n_2}\rightarrow \U_{\lambda,R}^{\oplus n_1}\rightarrow L\rightarrow 0.$$
Then $\operatorname{Tor}_*^{\U_{\lambda,R}}(\M_{w,R},L)$ is the homology of the complex
$$\ldots\rightarrow\M_{w,R}^{\oplus n_2}\rightarrow \M_{w,R}^{\oplus n_1}\rightarrow 0.$$
The individual terms of this complex are annihilated by $\J_{w,R}$ hence so is the homology.
This proves (1).
%(1) follows because the Tor's are computed as the homology of  a complex whose terms are
%direct sums of several copies of $\M_{w,R}$.

Let us prove (2). Since $\J_{d,R}$ annihilates $L$, we have $$\M_{w,R}\otimes^L_{\U_{\lambda,R}}L=
\left(\M_{w,R}\otimes^L_{\U_{\lambda,R}}(\U_{\lambda,R}/\J_{d,R})\right)\otimes^L_{\U_{\lambda,R}/\J_{d,R}}L.$$
Note that all simples $\M_{w'}$ appearing in $\operatorname{Tor}^*_{\U_\lambda}(\M_w,\U_\lambda/\J_d)$ satisfy
$w'\leqslant^L w, w'^{-1}\leqslant^L d$. This, together with (c6), implies (2). In (3), since $d\not\sim^L w^{-1}$,
we have $w'<^L w$. (3) follows.
\end{proof}

\subsection{Results on growth of lengths}
Recall that $\lambda\in \h^*_{\Q}$ is regular and $p$ is large enough, in particular, $\lambda$
is well-defined and regular mod $p$. Let $L$ be a simple in $\U_{\lambda,\F}^\chi\operatorname{-mod}$.

Given $m$, we always choose $p$ large enough for $V(\mu)$ to be irreducible mod $p$
provided $\langle\rho^\vee,\mu\rangle\leqslant m$ and $V^m_\F\otimes V^m_\F$ to be
semisimple.

Let us write $\operatorname{pr}_\lambda$ for the  projection $\U_\F\operatorname{-mod}^{\chi}\twoheadrightarrow
\U_\F\operatorname{-mod}_{\lambda}^\chi$. For a module $M\in \U_\F\operatorname{-mod}^\chi$,
let us write $\ell(M)$ for its length. We want to understand the behavior of the length  $\ell(\operatorname{pr}_\lambda(V^m_\F\otimes L))$ as a function of $m$.

\begin{Prop}\label{Prop:char_p_length}
Let $L$ be an irreducible $\U^\chi_{\lambda,\F}$-module. Let $\J\subset \U_\lambda$ be a maximal  primitive ideal
with the following property: $L$ is annihilated by $\J_R$. Let $\overline{\Orb}$ be the associated variety
of $\U_\lambda/\J$. Then there are real numbers  $c,C$
with $0<c<C$ such that for all $m\in \Z_{>0}$ and $p\gg m$ we have
$$c<\frac{\ell(\operatorname{pr}_\lambda(V^m_{\F}\otimes L))}{m^{\dim \Orb}}<C.$$
\end{Prop}

We will deduce this from an analogous result for Harish-Chandra bimodules. Namely, let $\mathcal{M}\in \,^\infty_{\lambda}\HC^1_{\lambda}$. Then we can consider the HC bimodule
$\operatorname{pr}_\lambda(V^m\otimes \mathcal{M})\in \,^\infty_{\lambda}\HC^1_{\lambda}$  and  its length
$\ell(\operatorname{pr}_\lambda(V^m\otimes \mathcal{M}))$.

\begin{Prop}\label{Prop:HC_length}
Let $\M\in \,_{\lambda}^\infty\operatorname{HC}^1_{\lambda}$ with $\VA(\M)=\overline{\Orb}$.
Let $m\in \Z_{>0}$. Then there are $0<c<C$  such that for all $m\in \Z_{>0}$ we have
$$c<\frac{\ell(\operatorname{pr}_\lambda(V^m\otimes \M))}{m^{\dim\Orb}}<C.$$
\end{Prop}

\subsection{Lengths for HC bimodules}\label{SS_HC_lengths}
In this section we will prove Proposition \ref{Prop:HC_length} and work over $\C$.
We are going to bound $\ell(\operatorname{pr}_\lambda(V^m\otimes \mathcal{M}))$
by two degree $\dim \Orb$ polynomials in $m$, where $\VA(\mathcal{M})=\overline{\Orb}$.

Under the  Bernstein-Gelfand equivalence  $\,_{\lambda}^\infty\operatorname{HC}^1_{\lambda}
\xrightarrow{\sim} \mathcal{O}(\mu)$ (in the notation of Section \ref{SS_HC_prim})
$\U_\lambda$ maps to the indecomposable projective $\Delta(w_0\mu)$. Every indecomposable projective
in $\,_{\lambda}^\infty\operatorname{HC}^1_{\lambda}$ appears as a summand in an object of the form $\operatorname{pr}_\lambda(V_0\otimes \U_\lambda)$ for a suitable finite dimensional $G$-module $V_0$
that we fix from now on.

\begin{Lem}\label{Lem:length_growth}
Let $\mathcal{M}\in \,_{\lambda}^\infty\operatorname{HC}(\U)^1_{\lambda}$   have associated variety $\overline{\Orb}$.
Then $$\dim \operatorname{Hom}_{bimod}(\U_\lambda, V^m\otimes \M)$$
is a degree $\dim \Orb$ polynomial in $m$.
\end{Lem}
\begin{proof}
Note that $\Hom_{bimod}(\U_\lambda,V^m\otimes\mathcal{M})=\Hom_{bimod}(U(\g), V^m\otimes \mathcal{M})$
because  $V^m\otimes \mathcal{M}$ has genuine central character $\lambda$ on the right.
Also $\Hom_{bimod}(U(\g), V^m\otimes \mathcal{M})=(V^m\otimes \mathcal{M})^G=\Hom_{G}(V^{m*},\mathcal{M})$.
So $$\dim \operatorname{Hom}_{bimod}(\U_\lambda, V^m\otimes \M)=\dim \Hom_G(V^{m*},\M).$$
Set $\mathcal{M}':=\gr\mathcal{M}$ with respect to some good filtration, this is a finitely
generated $G$-equivariant $\C[\g]$-module. Clearly, $\dim \Hom_G(V^{m*},\M)=\dim \Hom_G(V^{m*},\M')$.
Consider the filtration $\mathcal{M}'_{\leqslant i}$ on
$\mathcal{M}'$ given by $\mathcal{M}'_{\leqslant i}$ being the sum of the isotypic components of
$\mathcal{M}'$ with $\langle \rho^\vee,\mu\rangle\leqslant i$. This filtration is compatible with
the similarly defined filtration on $\C[\g]$.

It is well known that for any finitely
generated commutative $G$-algebra  $A$, the algebra $\gr A$ (for the filtration $A=\bigcup_i A_{\leqslant i}$)
is finitely generated and
for any finitely generated $G$-equivariant $A$-module $M$, the $\gr A$-module $\gr M$ is finitely
generated. This is because $\gr A=(\C[G/U\times G/U]^T\otimes A)^G$ and a similar equality
holds for $\gr M$, here $U$ is a maximal unipotent subgroup of $G$.

It follows that the GK dimensions
of $M,\gr M$ are the same. Since $\dim \Hom_G(V_m^*,\mathcal{M}')=\dim \mathcal{M}'_{\leqslant m}=
\dim \gr\mathcal{M}'_{\leqslant m}$, the left hand side of this equality is the Hilbert polynomial
of the graded module $\gr\mathcal{M}'$. But the GK-dimension of $\mathcal{M}'$ is $\dim \Orb$
and our claim follows.
%It is well-known that $\gr \mathcal{M}'$ with
%respect to the filtration $\mathcal{M}'_{\leqslant i}$ is still  a finitely generated
%$\C[\g]$-module.  So it has the same support as $\mathcal{M}'$, i.e., $\overline{\Orb}$.
%The claim of the lemma follows because the sum we need to compute is the Hilbert polynomial of $\gr\mathcal{M}'$.
\end{proof}

\begin{proof}[Proof of Proposition \ref{Prop:HC_length}]
Let us prove that $\ell(\operatorname{pr}_\lambda(V^m\otimes \mathcal{M}))\geqslant Q(m)$, where $Q$ is a degree $\dim \Orb$ polynomial. Note that $\dim\Hom_{bimod}(\U_\lambda, V^m\otimes \mathcal{M})$ is the multiplicity
of the simple bimodule covered by $\U_\lambda$ in $V^m\otimes \mathcal{M}$. By Lemma
\ref{Lem:length_growth}, this multiplicity is a degree $\dim \Orb$ polynomial in $m$, say $Q(m)$.
On the other hand, it is clear that $\ell(\operatorname{pr}_\lambda(V^m\otimes \mathcal{M}))\geqslant
\dim\Hom_{bimod}(\U_\lambda, V^m\otimes \mathcal{M})=Q(m)$.

Let us prove that $\ell(\operatorname{pr}_\lambda(V^m\otimes \mathcal{M}))\leqslant \widetilde{Q}(m)$, where $\widetilde{Q}$ is also a degree $\dim \Orb$ polynomial.
Let $V_0$ be as in the second paragraph of Section \ref{SS_HC_lengths}. Then
\begin{align*}\ell(\operatorname{pr}_\lambda(V^m\otimes \mathcal{M}))
\leqslant&\dim\Hom_{bimod}(V_0\otimes \U_\lambda,
V^m\otimes\mathcal{M})\\&= \dim \Hom_{bimod}(\U_\lambda,V^m\otimes (V_0^*\otimes \mathcal{M}))=:\widetilde{Q}(m).
\end{align*}
Applying Lemma \ref{Lem:length_growth} to $V_0^*\otimes\mathcal{M}$, we see that $\widetilde{Q}(m)$
is a degree $\dim \Orb$
polynomial in $m$.
\end{proof}

In the proof of Proposition \ref{Prop:char_p_length} we will also need the following lemma.

\begin{Lem}\label{Lem:length_head}
There is a constant $0<c_0<1$ such that, for any object $\mathcal{M}\in \,_{\lambda}^\infty\operatorname{HC}(\U)^1_{\lambda}$ we have $\ell(\mathsf{head}\mathcal{M})\geqslant c_0 \ell(\mathcal{M})$.
\end{Lem}

%Here, as usual, $\mathsf{head}\mathcal{M}$ denotes the head of $\mathcal{M}$, i.e., the maximal semisimple quotient.

\begin{proof}
Note that $\,_{\lambda}^\infty\operatorname{HC}^1_{\lambda}$ is equivalent to the category of modules
over a finite dimensional algebra. We claim that in any such category
\begin{equation}\label{eq:length_ineq}\ell(\mathsf{head}(M))>\frac{1}{L}
\ell(M)\end{equation} for every module $M$, where $L$ is the maximum of lengths of the indecomposable projectives.
Suppose that we know (\ref{eq:length_ineq}) for all $M'$ with $\ell(M')<\ell(M)$. Pick a simple constituent $L$
in $\mathsf{head}(M)$. Let $P_L$ be the projective cover of $L$. We have a homomorphism $\varphi:P_L\rightarrow M$
whose composition with $M\twoheadrightarrow \mathsf{head}(M)$ coincides with $P_L\twoheadrightarrow L\hookrightarrow
\mathsf{head}(M)$. Clearly, $\mathsf{head}(\operatorname{coker\varphi})=\mathsf{head}(M)/L$. Applying
the induction hypothesis to $\operatorname{im}\varphi$ and $\operatorname{coker}\varphi$ we finish the
proof of this lemma.
\end{proof}

\subsection{Lengths in characteristic $p$}
To prove Proposition \ref{Prop:char_p_length} we will need the following  technical lemma.

\begin{Lem}\label{Lem:product_length}
Let $L$ be an irreducible $\U^\chi_{\lambda,\F}$-module
such that $\J=\J_{w^{-1}}$ is a maximal primitive ideal with $\J_{R}L=0$. Then the following is true:
\begin{enumerate}
\item $\mathcal{M}_{w,R}\otimes_{\U_R}L\neq 0$,
\item there is $c_1>1$ independent of $p$
such that $\ell(\mathcal{M}_{w,R}\otimes_{\U_R}L)<c_1$.
\end{enumerate}
\end{Lem}

We will first deduce Proposition \ref{Prop:char_p_length} from this lemma and then prove it.

\begin{proof}[Proof of Proposition \ref{Prop:char_p_length}]
Let $\J$ be a maximal primitive ideal in $\U_\lambda$ such that $L$ is annihilated by $\J_R$.
By (c3), $\operatorname{pr}_\lambda(V^m_R\otimes_R L)=\operatorname{pr}_\lambda(V^m\otimes \U/\J)_{R}\otimes_{\U_R}L$.
Thanks to Proposition \ref{Prop:HC_length}, what we need to prove is that there are constants
$0<c<C$ such that $c\leqslant\ell(\mathcal{M}_{\F}\otimes_{\U_\F}L)/\ell(\mathcal{M})\leqslant C$ for any
$\M\in \,_{\lambda}^\infty\operatorname{HC}(\U_R)^1_{\lambda}$ whose right annihilator is $\J_R$.
By (1) of Lemma \ref{Lem:product_length} combined with Lemma \ref{Lem:length_head} and (c4), we can set
$c:=c_0$ from Lemma \ref{Lem:length_head}. By Lemma \ref{Lem:product_length}, we can set $C:=c_1$.
\end{proof}

\begin{proof}[Proof of Lemma \ref{Lem:product_length}]
Let us prove (1). Pick $w'$ such that $\M_{w'}\otimes_{\U_\lambda}\M_w$ has $\M_d$ as a direct summand in
$\HC_{\dcell}(\U_\lambda)$ ($w'$ exists because $\HC^{ss}_{\dcell}(\U_\lambda)$ is a rigid monoidal
category). Set $w_1=w', w_2=w$ in (c2) and let $M_1\subset M_2$ be as in (c2).
The equality $\M_{w,R}\otimes_{\U_R}L\neq 0$ will follow once we show that
$(\left(\M_{w',R}\otimes_{\U_R}\M_{w,R}\right)/M_1)\otimes_{\U_{R}}L\neq 0$.

First, let us check that
$\M_{d,R}\otimes_{\U_{R}}L\neq 0$. By the choice of $d$, we have $(\U_{\lambda,R}/\J_{d,R})\otimes_{\U_{\lambda,R}}L\neq 0$. So $\M_{d,R}\otimes_{\U_{\lambda,R}}L\neq 0$
as long as $\operatorname{Tor}^i_{\U_{\lambda,R}}((\U_{\lambda,R}/\J_{d,R})/M_{d,R},L)$ does not have
$L$ in its Jordan-Hoelder series for $i=0,1$. This is a consequence of (3) Lemma \ref{Lem:HC_der_prod}.
 We conclude that $\M_{d,R}\otimes_{\U_{\lambda,R}}L\neq 0$.

From here we deduce that $(M_2/M_1)\otimes_{\U_{\lambda,R}}L\neq 0$. Similarly to the previous paragraph
this implies  $(\left(\M_{w',R}\otimes_{\U_R}\M_{w,R}\right)/M_1)\otimes_{\U_{\lambda,R}}L\neq 0$. This finishes the proof of (1).

%Recall that the semisimple part of the subquotient
%of $\,_{\lambda}^\infty\operatorname{HC}(\U)^1_{\lambda}$ corresponding to $\Orb_{\dcell}$ is a rigid
%monoidal category. So let $\mathcal{M}^*$ be the simple dual to $\mathcal{M}$ in that semisimple part.
%So $\mathcal{M}_0$ appears as the direct summand in the image of $\mathcal{M}^*\otimes_{\U} \mathcal{M}$
%in the subquotient corresponding to $\Orb$. It follows that there is a quotient $\tilde{\mathcal{M}}_0$
%of $\mathcal{M}^*\otimes_{\U} \mathcal{M}$ such that $\mathcal{M}_0\hookrightarrow \tilde{\mathcal{M}}_0$
%and $\tilde{\mathcal{M}}_0/\mathcal{M}_0$ is supported on $\partial \Orb_{\dcell}$. We claim
%that $\tilde{\mathcal{M}}_{0,\F}\otimes_{\U_\F}L\neq 0$, this will imply that $\mathcal{M}_\F\otimes_{\U_\F}L\neq 0$.
%This follows from the following claim: if $\mathcal{M}'$ is a HC bimodule supported on
%$\partial\Orb_{\dcell}$, then $L$ is not a constituent of $\operatorname{Tor}^*_{\U_\F}(\mathcal{M}'_\F,L)$.
%Indeed, this Tor admits a filtration whose successive quotients are annihilated by primitive
%ideals strictly containing $\J$.

Let us prove (2).
It is enough to prove this statement with $\M_{w,R}$ replaced with a bimodule that covers
it, e.g., $\operatorname{pr}_\lambda(V_{0,R}\otimes \U_{\lambda,R})_R$, where $V_0$ is as
in the beginning of Section \ref{SS_HC_lengths}. Note that bimodule is projective as a right module.
On the level of $K_0$ the operator $\operatorname{pr}_\lambda(V_{0,R}\otimes \U_{\lambda,R})\otimes_{\U_{\lambda,R}}\bullet$ is the multiplication by some element, say $y$, of $\C W_{[\lambda]}$ independent of $p$. For $b\in \mathfrak{B}$, we can expand $yb=\sum_{b'\in \mathfrak{B}}m_{bb'}b'$. Then $c_1=\max_{b\in \mathfrak{B}}\sum_{b'}m_{bb'}$ satisfies the conditions of (2).
\end{proof}

\section{Proof of Theorem \ref{Thm:main}}
In this section we will prove Theorem \ref{Thm:main}. Part (1) is proved in Section
\ref{SS_part1_proof}, while the proof of part (2) occupies the remainder of the section.
We will describe the main steps of the proof in  Section \ref{SS_proof_2_outline}.

\subsection{Proof of part (1) of Theorem \ref{Thm:main}}\label{SS_part1_proof}
%Note that the category of HC bimodules with generalized central character $\lambda$ on both sides
%is defined over a finite localization of $\Z$ and so acts on $\U_{\F}\operatorname{-mod}_{\lambda,\chi}$.

For a two-sided cell $\dcell$ in $W_{[\lambda]}$ consider the full subcategory
$\U_{\lambda,\F}\operatorname{-mod}_{\leqslant \dcell}^{\chi}$ that is the Serre span of all simples annihilated by
$\J_{w,R}$ with $w\in \dcell$.  Note that $D^b(\U_{\lambda,\F}\operatorname{-mod}^\chi)_{\leqslant \dcell}$ is a submodule category for the action of $D^b_{HC}(\U_{\lambda,R}\operatorname{-mod})$, this follows from
(2) of Lemma \ref{Lem:HC_der_prod}.

We have the following result.

\begin{Prop}\label{Prop:categor_filtration}
All irreducible representations of $W_{[\lambda]}$ occurring in $$K_0(\U_{\lambda,\F}\operatorname{-mod}_{\leqslant \dcell}^{\chi})/K_0(\U_{\lambda,\F}\operatorname{-mod}_{<\dcell}^{\chi})$$ belong to the two-sided cell
$\dcell$.
\end{Prop}
\begin{proof}
We need to prove two statements:
\begin{enumerate}
\item $\C W_{[\lambda],<\dcell}K_0(\U_{\lambda,\F}\operatorname{-mod}^\chi_{\leqslant \dcell})
\subset K_0(\U_{\lambda,\F}\operatorname{-mod}^\chi_{<\dcell})$.
\item  $\C W_{[\lambda],\leqslant \dcell} K_0(\U_{\lambda,\F}\operatorname{-mod}^\chi_{\leqslant \dcell})=
K_0(\U_{\lambda,\F}\operatorname{-mod}^\chi_{\leqslant \dcell})$,
\end{enumerate}
(1) follows from (1) of Lemma \ref{Lem:HC_der_prod}. Let us prove (2).
Let $L\in \U_{\lambda,\F}\operatorname{-mod}^\chi_{\leqslant \dcell}$ be a simple object
annihilated by $\J_{d,R}$, where $d\in \dcell$. Tautologically $(\U_{\lambda,R}/\J_{d,R})\otimes^L_{\U_{\lambda,R}/\J_{d,R}}L=L$. Let us write
$V$ for $\operatorname{Span}_{w\leqslant^L d}(\M_{w^{-1}})$. Note that
$\C W_{[\lambda],\leqslant \dcell}V=V$.  Indeed, all $W$-irreducibles appearing in $V$
belong to the families indexed by two-sided cells $\leqslant \dcell$.

So we have $[\U_{\lambda}/\J_{d}]=
\sum c_{w_1,w_2}[\M_{w_1}][\M_{w_2}]$ (an equality in $V$), where $w_1$ runs over $\dcell$ and $w_2$
over the elements such that $w_2^{-1}\leqslant^L d$.
It follows that $[L]$ belongs to the linear span of the classes of the form
$$[\left(\M_{w_1,R}\otimes^L_{\U_\lambda,R}\M_{w_2,R}\right)_{\U_{\lambda,R}/\J_{d,R}}L]=
[\M_{w_1}][\M_{w_2,R}\otimes^L_{\U_{\lambda,R}}L].$$
Note that $\M_{w_2,R}\otimes^L_{\U_{\lambda,R}}L\in D^b(\U_{\lambda,\F}\operatorname{-mod}^\chi)_{\leqslant \dcell}$
by (1) of Lemma \ref{Lem:HC_der_prod}. We deduce that $[L]\in
\C W_{[\lambda],\leqslant \dcell}K_0(\U_{\lambda,\F}\operatorname{-mod}^\chi_{\F})$.
This implies (2).
%(2) easily follows from Lemma \ref{Lem:HC_der_prod}. Let us prove (1). As in the proof of (2) of Lemma
%\ref{Lem:HC_der_prod}, $\M_{w,R}\otimes^L_{\U_{\lambda,R}}L=(\M_{w,R}\otimes^L_{\U_{\lambda,R}}\U_{\lambda,R}/\J_R)
%\otimes_{\U_{\lambda,R}/\J_R}L$. Since the two-sided cells corresponding to $\J$ is contained in $\dcell$, we
%see that $\M_{w}\otimes^L_{\U_\lambda}\U_\lambda/\J$ lies in $D^b_{HC}(\U_\lambda\operatorname{-bimod})_{\leqslant %\dcell}$. This implies (1).
%We can consider $K_0(\U_{\lambda,\F}\operatorname{-mod}_{\chi})$ as a $\Jalg(W_{[\lambda]})$-module
%via the isomorphism $\Jalg(W_{[\lambda]})\xrightarrow{\sim}\C W_{[\lambda]}$. Recall that, under the identification,
%$\Jalg(W_{[\lambda]})=\bigoplus_{\dcell} K_0(\HC^{ss}_{\dcell}(\U_\lambda))\cong
%K_0(\HC(\U_\lambda))\cong \C W_{[\lambda]}$ this isomorphism sends $[\M_w]$ to
%$\sum_d [\M_w\otimes^L_{\U_\lambda}\M_d]$.
%It follows that
%for $w\in \dcell$ and an object $N\in \U_{\lambda,\F}\operatorname{-mod}^{\leqslant\dcell}_{\chi}$,
%we have $t_w[N]=[\mathcal{M}_{w,R}\otimes_{\U_{\lambda,R}} N]$
%modulo $\U_{\lambda,\F}\operatorname{-mod}^{<\dcell}_{\chi}$\footnote{Is this clear, or do I need to say more?}.
%Similarly to the proof of (1) of Lemma \ref{Lem:product_length},
%we see that the representation of $\Jalg(W_{[\lambda]})$ in
%$K_0(\U_{\lambda,\F}\operatorname{-mod}^{\leqslant \dcell}_{\chi})/
%K_0(\U_{\lambda,\F}\operatorname{-mod}^{< \dcell}_{\chi})$ factors through the direct
%summand $\Jalg_{\dcell}(W_{[\lambda]})$. This implies the claim of the proposition.
\end{proof}

\begin{proof}[Proof of (1) of Theorem \ref{Thm:main}]
Proposition \ref{Prop:categor_filtration} and an easy induction on $\dcell$
show that $$K_0(\U_{\lambda,\F}\operatorname{-mod}^\chi)_{\leqslant\dcell}=
K_0(\U_{\lambda,\F}\operatorname{-mod}^\chi_{\leqslant \dcell}).$$
This establishes the claim of  part (1) at $q=1$.

To prove the full claim one uses the graded lifts mentioned in Section
\ref{SS_loc_char_p}.  Namely, let us write $\mathcal{C}^{gr}$ for the
graded lift of $\mathcal{C}:=\U_{\lambda,\F}\operatorname{-mod}^\chi$.
%and let $\mathfrak{H}_{[\lambda]}$ be the subcategory in the affine Hecke category
%corresponding to $W_{[\lambda]}$.
 We still have the two-sided cell filtration $\mathcal{C}^{gr}_{\leqslant \dcell}$ on $\mathcal{C}^{gr}$
that is closed under the grading shifts and lifts the filtration
$\U_{\lambda,\F}\operatorname{-mod}^{\chi}_{\leqslant\dcell}$.

We claim that  $D^b(\mathcal{C}^{gr}_{\leqslant \dcell})$ is invariant
under the braid group $B_{[\lambda]}$, hence its Grothendieck
group is invariant under the corresponding Hecke algebra.
It is enough to check that, for a simple module $L\in \mathcal{C}^{gr}_{\leqslant \dcell}$
and a generator $\tilde s_\alpha\in B_{[\lambda]}$, we have
$\tilde s_\alpha(L)\in D^b(\mathcal{C}^{gr}_{\leqslant \dcell})$.
In view of Lemma \ref{either_or} this follows from $s_\alpha([L])\in K^0(\mathcal{C}_{\leqslant \dcell})$,
which has already been proven.

So $K_0(\mathcal{C}^{gr}_{\leqslant \dcell})/K_0(\mathcal{C}^{gr}_{<\dcell})$ is an
$\mathcal{H}_q(W_{[\lambda]})$-module flat over $\C[q^{\pm 1}]$ that specializes to
$$K_0(\U_{\lambda,\F}\operatorname{-mod}^\chi_{\leqslant \dcell})/K_0(\U_{\lambda,\F}\operatorname{-mod}^\chi_{<\dcell})$$
at $q=1$. The latter factors through the quotient corresponding to $\dcell$ hence so is the former.
\end{proof}

\subsection{Outline of the proof of (2) of Theorem \ref{Thm:main}}\label{SS_proof_2_outline}
Below we will prove part (2) of Theorem \ref{Thm:main}. We will start with the
$\chi=0$ case.

\begin{Prop}\label{Prop:0_dim_growth}
(2) of Theorem \ref{Thm:main} holds when $\chi=0$.
\end{Prop}

Here we will use an easy adaptation of an argument due to Etingof that
relates the degrees of dimension polynomials with the GK dimensions of simples in
the category $\mathcal{O}$ (in characteristic $0$). We will reduce the case of general $\chi$ to $\chi=0$
by using the degeneration map $K_0(\U_{\lambda,\F}^\chi\operatorname{-mod})\rightarrow
K_0(\U_{\lambda,\F}^0\operatorname{-mod})$. This map can be shown to be independent of $p$.
Hence it preserves the dimension polynomials.   The  most nontrivial
step is to show that the degeneration of a simple module that lies in $K_0(\U^\chi_{\lambda,\F}\operatorname{-mod})_{\leqslant \dcell}$ but not in the lower filtration terms
lies in $K_0(\U^0_{\lambda,\F}\operatorname{-mod})_{\leqslant \dcell}$ (this is straightforward) but not in
the lower filtration terms (this is harder, we need  Proposition \ref{Prop:char_p_length} to handle this part).

Before we proceed to proving part (2), let us reformulate it. For this we need the following lemma.

\begin{Lem}\label{Lem:prim_annihilator}
Let a simple $L\in \U_{\lambda,\F}^\chi\operatorname{-mod}_{\leqslant \dcell}$ but not in lower filtration terms.
Then there is a unique maximal primitive ideal $\J=\J_d\in \operatorname{Prim}(\U_\lambda)$ with
$\J_R L=0$. We have $d\in \dcell$.
\end{Lem}
\begin{proof}
Let $\J_1,\J_2$ be two maximal primitive ideals with the required property.
Assume, in addition, that the two-sided cell corresponding to $\J_1$ is maximal
possible.  By (c6), $\U_{\lambda,R}/(\J_{1,R}+\J_{2,R})$ is filtered by $\M_{w,R}$
for $w$ lying in two-sided cells smaller than $\dcell$. If $\M_{w,R}\otimes_{\U_{\lambda,R}}L\neq 0$,
then $\J_{w^{-1},R}L=0$. The two-sided cell of $\J_{w^{-1}}$ is strictly smaller than those of
$\J_1,\J_2$. A contradiction. The inclusion $d\in \dcell$ follows from the definition
of   $\U_{\lambda,\F}^\chi\operatorname{-mod}_{\leqslant \dcell}$.
\end{proof}

\begin{Rem}
So (2) of Theorem \ref{Thm:main} says that the dimension polynomial of $L$  has degree
equal to $\frac{1}{2}\operatorname{GK-}\dim (\U_\lambda/\J)$, where $\J$ is the maximal
primitive ideal in $\U_\lambda$ such that $\J_{R}L=0$. This statement makes
sense for other classes of quantizations, e.g. for those of symplectic singularities
and we expect it to hold in this setting.
\end{Rem}

\subsection{Etingof's construction}\label{SS_Etingof}
We will prove Proposition \ref{Prop:0_dim_growth} by adapting an argument due to Etingof
from categories $\mathcal{O}$ for type A rational Cherednik algebras to BGG categories $\mathcal{O}$.

Let us explain this argument. Let $R$ be a finite localization of $\Z$ such that $\lambda\in \h^*_R$. Then we can consider the Verma module $\Delta_R(\lambda)$ with highest weight $\lambda+\rho$. Note that $\Delta_R(\lambda)$ is naturally graded
with highest vector in degree $0$ and the operators $f_\alpha$ have degree $1$ for a simple root $\alpha$.
%This module has a unique (up to scaling
%by elements of $R^\times$) contravariant form $B_\lambda$.
%and different graded components are orthogonal with respect to $B_{\lambda}$.
Let $p$ be a prime number invertible in $R$. Set $\Delta_{\F_p}(\lambda):=\F_p\otimes_{R}\Delta_R(\lambda),
\Delta_{\mathbb{Q}}(\lambda):=\mathbb{Q}\otimes_R\Delta_R(\lambda)$. Let $L_{\mathbb{Q}}(\lambda)$ be the unique
irreducible quotient of $\Delta_{\mathbb{Q}}(\lambda)$. The module $\Delta_{\F_p}(\lambda)$ has a unique
graded simple quotient, let us denote it by $L_{\F_p}(\lambda)$. The modules $L_{\F_p}(\lambda')$ for
$\lambda'\in W\lambda$ are absolutely irreducible and  pairwise non-isomorphic and so their base changes to $\F$ form a complete collection of the irreducibles in $\mathcal{U}_{\lambda,\F}^0\operatorname{-mod}$.
Note that $L_{\mathbb{Q}}(\lambda)$ and $L_{\F_p}(\lambda)$ are graded quotients. Let $L^i_{\mathbb{Q}}(\lambda),L^i_{\F_p}(\lambda)$
denote the $i$th graded component.

The following lemma is due to Etingof (in the Cherednik case).

\begin{Lem}\label{Lem:dim_growth}
Suppose that $p$ is sufficiently large. Then there is a positive integer $N_\lambda$ independent of $p$ such that $\dim L^i_{\F_p}(\lambda)=
\dim L^i_{\mathbb{Q}}(\lambda)$ for $i< p/N_\lambda$.
\end{Lem}
\begin{proof}
Recall that, for any $\lambda'\in \h^*_R$, the module $\Delta_R(\lambda')$
has a unique (up to scaling by elements of $R$) contravariant form $B_{\lambda'}$,
and different graded components are orthogonal with respect to $B_{\lambda'}$. We  assume that
$B$ is nondegenerate on the highest weight component.

We can also consider a one-parameter
deformation $\Delta_{R[t]}(\lambda+t\rho)$, where $t$ is an independent variable. The module
$\Delta_{R[t]}(\lambda+t\rho)$ comes with a contravariant
form $B_{\lambda+t\rho}$. Let us write $B^i_\lambda$ for the restriction of $B_\lambda$ to $\Delta^i_R(\lambda)$.
The notation $B^i_{\lambda+t\rho}$ has the similar meaning.

We note that $L_{\mathbb{Q}}(\lambda)$ is the quotient of $\Delta_{\mathbb{Q}}(\lambda)$ by the radical of $B_{\lambda,\mathbb{Q}}$, the specialization of $B_\lambda$ to $\Q$ and, similarly, $L_{\F_p}(\lambda)$ is the quotient of $\Delta_{\F_p}(\lambda)$ by the radical of $B_{\lambda,\mathbb{F}_p}$.   So we need to check that for $i<p/N_\lambda$, the radicals of $B^i_{\lambda,\mathbb{F}_p}$ and of $B^i_{\lambda,\mathbb{Q}}$ have the same dimension.

Consider the finitely generated $\mathbb{K}[t]$-module $\Delta^i_{\mathbb{K}[t]}(\lambda+t\rho)$, where
$\mathbb{K}=\mathbb{Q}$ or $\F_p$. The form $B^i_{\lambda+t\rho, \mathbb{K}}$ defines a descending filtration
$F^j\Delta^i_{\mathbb{K}}(\lambda)$ (the Jantzen filtration): by definition, $F^0\Delta^i_{\mathbb{K}}(\lambda)=
\Delta^i_{\mathbb{K}}(\lambda)$ and $F^j\Delta^i_{\mathbb{K}}(\lambda)$ is the radical of the well-defined form
$t^{-j}B^i_{\lambda+t\rho, \mathbb{K}}|_{t=0}$ on $F^{j-1}\Delta^i_{\mathbb{K}}(\lambda)$. It is clear that
$\dim F^j\Delta^i_{\mathbb{Q}}(\lambda)\geqslant \dim F^j\Delta^i_{\F_p}(\lambda)$ for all $p$. The equality
$\dim F^j\Delta^i_{\mathbb{Q}}(\lambda)= \dim F^j\Delta^i_{\F_p}(\lambda)$ for all $j$ (that will immediately
imply what we need, which is the $j=1$ case) for $i<p/N_\lambda$ will follow if we check that
the order of vanishing of $f^i_p(t):=\det B^i_{\lambda+t\rho,\F_p}$ at $t=0$ coincides with the order
of vanishing of $f^i(t):=\det B^i_{\lambda+t\rho,\mathbb{Q}}$ at $t=0$ -- for example, the former
equals $\sum_j j\dim F^j\Delta^i_{\F_p}(\lambda)$. Clearly, $f^i_{p}(t)$ is obtained from $f(t)$
by reduction mod $p$.

The polynomial $f^i(t)$ can be decomposed as $C^it^{n_0}\prod_{j=1}^k (t-z^i_j)^{n_j}$, where $C^i\in R$ and
$z^i_j$ are nonzero elements that
lie, a priori, in the algebraic closure of $\mathbb{Q}(t)$. In fact, $f^i(z)=0$ means that
there is a singular vector in $\Delta^{i'}_{\lambda+z\rho}(\lambda)$, where $0<i'\leqslant i$.
In particular, there is $w\in W, w\neq 1,$ such that $\langle \lambda+z\rho-w(\lambda+z\rho),\rho^\vee\rangle=i'$.
This is equivalent to
\begin{equation}\label{eq:gamma}
z=\langle \rho-w\rho,\rho^\vee\rangle^{-1}(i'-\langle \lambda- w\lambda, \rho\rangle)
\end{equation}
In particular, $f^i(z)=0$ implies $z\in R$.
%, i.e., $\langle \lambda+z \rho,\alpha^\vee\rangle\in \Z_{>0}$
%for some positive coroot $\alpha^\vee$. In particular, $z$ has the form $\frac{z}{\langle\rho,\alpha^\vee\rangle}-
%\langle\lambda,\alpha^\vee\rangle$, where $z\in \Z_{>0}$ so $z\in R$.

Therefore what we need
to prove is that, for $i<p/N_\lambda,$ we have $C^i,z^i_j\neq 0$ modulo $p$.

Let us show that $C^i$ is nonzero mod $p$.
Consider the baby Verma module $\underline{\Delta}_{\F_p[t]}(\lambda+t\rho)=\Delta_{\F_p[t]}(\lambda+t\rho)/
(\mathfrak{n}^-_{\F_p})^{(1)}\Delta_{\F_p[t]}(\lambda+t\rho)$. Note that $\underline{\Delta}^i_{\F_p[t]}(\lambda+t\rho)=
\Delta^i_{\F_p[t]}(\lambda+t\rho)$ as long as $i<p$. It is well-known that for a generic $z\in \F$ the module
$\underline{\Delta}^i_{\F}(\lambda+z\rho)$ is irreducible. In particular, $f^i_p(t)\neq 0$ as a polynomial
so $C^i\neq 0$ mod $p$.

Now let us show that $z^i_j\neq 0$ mod $p$ using (\ref{eq:gamma}).
Recall that $x$ stands for the least common multiple of the denominators of the coordinates
of $\lambda$ in the basis of simple roots. Then $xi'-x\langle \lambda-w\lambda,\rho\rangle$
is a nonzero multiple of $p$. Clearly as long as $p$ is large enough, there is $N_\lambda$
such that for $i< p/N_{\lambda}$, we have $0<xi'-x\langle \lambda-w\lambda,\rho\rangle<p$.
\end{proof}

\subsection{Proof of Proposition \ref{Prop:0_dim_growth}}
Let us deduce Proposition \ref{Prop:0_dim_growth} from Lemma \ref{Lem:dim_growth}.

\begin{proof}[Proof of Proposition \ref{Prop:0_dim_growth}]
The GK dimension of $L(\lambda)$ equals to $\frac{1}{2}\operatorname{GK-}\dim(\U_\lambda/\operatorname{Ann}L(\lambda))$
and the latter coincides with $\dim\Orb_{\dcell}/2$. Therefore $\sum_{j=0}^i \dim L^j(\lambda)$ is a polynomial
in $i$ of degree $\dim\Orb_{\dcell}/2$. So $\dim L_{\F}(\lambda)\geqslant \sum_{j=0}^i \dim L^j_{\F_p}(\lambda)
=\sum_{j=0}^i \dim L^j(\lambda)$, where the equality follows from Lemma \ref{Lem:dim_growth}, as long as $i< p/N_\lambda$. This shows that $\dim L_{\F}(\lambda)$ is bounded below by a degree $\dim \Orb_{\dcell}/2$ polynomial in $p$.

Now let us show that $\dim L_\F(\lambda)$ is bounded from above by a degree $\dim \Orb_{\dcell}/2$ polynomial
in $p$. Note that $L_\F(\lambda)$ is a quotient of the baby Verma module $\underline{\Delta}_\F(\lambda)$.
We have $\underline{\Delta}^j_\F(\lambda)=0$ for $j\geqslant 2\langle\rho,\rho^\vee\rangle p$ hence
$L^j_\F(\lambda)=0$ for such $j$. We claim that, for any $m_1>1>m_2$, there is a constant $M$ such that
\begin{equation}\label{eq:dim_bound}\sum_{j=0}^{\lfloor m_1p\rfloor}\dim L^j_\F(\lambda)\leqslant
M \sum_{j=0}^{\lfloor m_2p\rfloor } \dim L^j_\F(\lambda), \end{equation}
this will establish the upper bound thanks to Lemma \ref{Lem:dim_growth}.

Let us write $U_{\leqslant k}(\mathfrak{n}^-_{\F})$ for $k$th filtration term with respect to the PBW filtration on $U(\mathfrak{n}^-_\F)$
and $U^{\leqslant i}(\mathfrak{n}^-_{\F})$ for $\bigoplus_{j=0}^i U^{j}(\mathfrak{n}^-_\F)$.
For any $i$, we have
\begin{equation}\label{eq:filtr_compat} \bigoplus_{j=0}^i L_\F^j(\lambda)=U^{\leqslant i}(\mathfrak{n}^-_{\F}) L_\F^0(\lambda).
\end{equation}
Note that the filtrations
$U^{\leqslant i}(\mathfrak{n}^-_{\F})$ and $U_{\leqslant k}(\mathfrak{n}^-_\F)$ are compatible in the sense that there are
constants $c_1<1<c_2$ such that $U_{\leqslant c_1 i}(\mathfrak{n}^-_\F)\subset \U^{\leqslant i}(\mathfrak{n}^-_\F)\subset
U_{\leqslant c_2 i}(\mathfrak{n}^-_\F)$. So, thanks to (\ref{eq:filtr_compat}), (\ref{eq:dim_bound}) will follow if we show that, for any $m\in \Z_{>1}$ there is $M$ such that for any $i$ we have
\begin{equation}\label{eq:dim_bound1}
\dim U_{\leqslant mi}(\mathfrak{n}^-_\F)\leqslant M\dim U_{\leqslant i}(\mathfrak{n}^-_\F).
\end{equation}
Let $x_1,\ldots,x_n$ be a basis of $\mathfrak{n}^-_{\F}$. (\ref{eq:dim_bound1}) will follow if we show
that every element of $U_{\leqslant mi}(\mathfrak{n}^-_\F)$ can be written as a sum of elements of
the form $PQ$, where $P$ is an ordered monomial in $x_1^{i+1},\ldots, x_n^{i+1}$ of degree $\leqslant m$ and
$Q$ is an element  of $U_{\leqslant i}(\mathfrak{n}^-_\F)$. The latter claim follows from the analogous one
on the associated graded level, which is straightforward.
\end{proof}

\begin{Rem}\label{upper_bound_alt}
Let us mention an alternative way to the prove the upper bound for  $\dim L_\F(\lambda)$
established above. Let $M$ be an object in category $\mathcal{O}$ over $\Ce$ of Gelfand-Kirillov dimension $d$.
We can find a $\g$-module $M_R$ is defined over $R$ with $M\cong M_R\otimes _R \Ce$,  and consider its based change $M_\F$ to a field of almost any prime characteristic. Let $\overline{M_\F}$ be the reduction of $M_\F$
by the zero $p$-central character. Then one can check that:

i) $\dim\overline{M_\F}=O(p^d)$.

ii)  The space  $K_0(\U_{\lambda,\F}\operatorname{-mod}_{\leqslant \dcell}^{\chi})$ is spanned
by classes $\dim\overline{L_\F}$, where $L$ runs over the set of irreducible module in category $\mathcal{O}$
belonging to cells $\dcell'\leqslant \dcell$.

These two properties clearly imply the upper bound.
\end{Rem}

\subsection{Degeneration map}
We have a one parameter subgroup $\gamma: \F^\times\rightarrow G_{\F}$ with $\gamma(t)\chi=t^{2}\chi$.
Via $\gamma$, the group $\F^\times$ acts on the sheaf of algebras $\U_{\F}|_{\F\chi}$, where the action on
the base $\F\chi$ is by dilations. This gives rise to the degeneration map $\delta: K_0(\U_\F^\chi\operatorname{-mod})
\rightarrow K_0(\U_\F^0\operatorname{-mod})$. Since $\F^\times$ acts trivially on the Harish-Chandra
center, we see that the map restricts to
\begin{equation}\label{eq:degener} \delta: K_0(\U_{\lambda,\F}^\chi\operatorname{-mod})
\rightarrow K_0(\U_{\lambda,\F}^0\operatorname{-mod}). \end{equation}

The following standard lemma summarizes basic properties of the degeneration map (\ref{eq:degener}).

\begin{Lem}\label{Lem:degener_properties}
The following are true.
\begin{enumerate}
\item Under the identifications $K_0(\U_{\lambda,\F}^\chi\operatorname{-mod})\cong
H_*(\mathcal{B}_e,\C), K_0(\U_{\lambda,\F}^0\operatorname{-mod})\cong
H_*(\mathcal{B},\C)$ the map $\delta$ coincides with the push-forward map
$H_*(\mathcal{B}_e,\C)\rightarrow H_*(\mathcal{B},\C)$. In particular, it is independent of $p$ and $W^a$-equivariant.
\item The map $\delta$ intertwines the endomorphisms $\operatorname{pr}_\lambda(V(\mu)_\F\otimes \bullet)$
for $p\gg \langle \rho^\vee,\mu\rangle$.
\item The map $\delta$ preserves the dimension polynomials.
\end{enumerate}
\end{Lem}
\begin{proof}
(2) and (3) are straightforward, let us prove (1). We can consider the categories
$\operatorname{Coh}(\mathcal{D}^\lambda_{\mathcal{B}_\F})^{t\chi}$ (here $t=0,1$) of all coherent sheaves of $\mathcal{D}^\lambda_{\mathcal{B}_\F}$-modules supported at the preimage of $t\chi$
and the corresponding derived category
$D^b(\operatorname{Coh}(\mathcal{D}^\lambda_{\mathcal{B}_\F}))^{t\chi}$. The derived
equivalence $R\Gamma$ induces  identifications $K_0(\operatorname{Coh}(\mathcal{D}^\lambda_{\mathcal{B}_\F})^{t\chi})
\cong K_0(\U_{\lambda,\F}\operatorname{-mod}^{t\chi})$ that are compatible with the degeneration maps.
Also note that the Chern character isomorphisms intertwine the degeneration maps. So it remains
to show that the identifications $K_0(\Coh(\mathcal{B}_{t\chi}^{(1)}))\xrightarrow{\sim}
K_0(\operatorname{Coh}(\mathcal{D}^\lambda_{\mathcal{B}_\F})^{t\chi})$ intertwine the
degeneration maps. This is a consequence of Lemma \ref{Lem:K_0_indep} and the projection formula.
\end{proof}

%Part (2) of Theorem \ref{Thm:main} will follow from Proposition \ref{Prop:0_dim_growth}, Lemma
%\ref{Lem:degener_properties} and the next result to be proved in the next section.

\subsection{Proof of (2) of Theorem \ref{Thm:main}}
We start with the following proposition.

\begin{Prop}\label{Prop:degen_cell_image}
Let $L$ be a simple in $\U_{\lambda,\F}^\chi\operatorname{-mod}_{\leqslant \dcell}$ but not in the
smaller filtration terms. Then the projection of $\delta([L])$ to the $\dcell$-isotypic component
of $K_0(\U_{\lambda,\F}^0\operatorname{-mod})_{\leqslant \dcell}$ is nonzero.
\end{Prop}
\begin{proof}
Assume the contrary. Let $\delta[L]=\sum_{i=1}^k [L_i]$, where $L_i$ are  simples in
$\U^0_{\lambda,\F}\operatorname{-mod}$. Note that $k$ is independent of $p$ by (1) of
Lemma \ref{Lem:degener_properties} and the fact that the basis of simples is independent of $p$.
By our assumption, $L_i\in \U_{\lambda,\F}^0\operatorname{-mod}_{<\dcell}$.

Pick a dominant weight $\mu$ and suppose that $p$ is very large. Let $M_\mu$
denote the functor $\operatorname{pr}_\lambda(V(\mu)_\F\otimes\bullet)$. By (2) of Lemma
\ref{Lem:degener_properties}, we have
\begin{equation}\label{eq:delta_commutes_w_prod}
\delta([M_\mu L])=\sum_{i=1}^k [M_\mu L_i].
\end{equation}

Note that for $N\in \U^\chi_{\F}\operatorname{-mod}_\lambda$, we have
\begin{equation}\label{eq:inequal} \ell([N])\leqslant \ell(\delta[N]).\end{equation}
By Proposition \ref{Prop:char_p_length} combined with Lemma \ref{Lem:prim_annihilator},
\begin{equation}\label{eq:sum1}\sum_{\mu, \langle\rho^\vee,\mu\rangle\leqslant m}\dim V(\mu)\ell(M_\mu L)\end{equation}
grows (with respect to $m$) faster than
\begin{equation}\label{eq:sum2}\sum_{i=1}^k\sum_{\mu, \langle\rho^\vee,\mu\rangle\leqslant m}\dim V(\mu)\ell(M_\mu L_i)\end{equation}
But combining (\ref{eq:delta_commutes_w_prod}) with (\ref{eq:inequal}), we see that
(\ref{eq:sum2})$\geqslant$(\ref{eq:sum1}). A contradiction.
\end{proof}

\begin{proof}[Proof of (2) of Theorem \ref{Thm:main}]
By (3) of Lemma \ref{Lem:degener_properties}, the map $\delta: K_0(\U_{\lambda,\F}^\chi\operatorname{-mod})
\rightarrow K_0(\U_{\lambda,\F}^0\operatorname{-mod})$ preserves the dimension polynomials.
For an irreducible $L\in \U_{\lambda,\F}^\chi\operatorname{-mod}_{\leqslant \dcell}$ (but not in
smaller filtration components), we have $\delta([L])=\sum_{i=1}^k [L_i]$, where $k$ is independent of $p$,
the simple $L_i$ belongs to a two-sided cell
$\dcell_i\leqslant \dcell$ and there is $i$ such that $\dcell_i=\dcell$. The dimension polynomial for
$L$ is the sum of the dimension polynomials for the $L_i$'s. Note that $\Orb_{\dcell_i}\subset \overline{\Orb}_{\dcell}$.
By Proposition \ref{Prop:0_dim_growth},
the dimension polynomial of $L$ has degree $\dim \Orb_{\dcell}/2$.
\end{proof}

\section{Application to W-algebras}
In this section we will use Theorem \ref{Thm:main} to prove conjectures from
\cite[Section 7.6]{LO} on the classification of finite dimensional irreducible
representations of W-algebras.

\subsection{Background on W-algebras}\label{SS_W_background}
Finite W-algebras (below we omit the adjective ``finite'') were introduced by Premet in
\cite{Premet1} (with alternative constructions later given by the second author).
These are associative algebras constructed from pairs $(\g,e)$, where $\g$ is a semisimple
Lie algebra over $\C$ and $e\in \g$ is a nilpotent element. Such  a W-algebra is a quantization
of the transverse Slodowy slice to the adjoint orbit $\Orb$ of $e$. The reader is referred to
the survey article  \cite{W_ICM} for details.

Let us recall Premet's definition. Include $e$ into an $\slf_2$-triple $(e,h,f)$. The element
$h$ induces the grading on $\g$ by eigenvalues of $\operatorname{ad}(h)$: $\g=\bigoplus_{i\in \Z}\g(i)$.
Let, as before, $\chi=(e,\cdot)$. The form $\omega(x,y)=\langle\chi,[x,y]\rangle$ is symplectic on
$\g(-1)$. Let us pick a lagrangian subspace $\ell\subset \g(-1)$. Form a subalgebra $\mathfrak{m}
\subset \g$ by $\mathfrak{m}=\bigoplus_{i\leqslant -2}\g(i)\oplus \ell$. Note that $\chi$
is the character of $\mathfrak{m}$ and that $\dim \mathfrak{m}=\frac{1}{2}\dim \Orb$, where
we write $\Orb$ for the orbit of $e$. Then, by definition, the W-algebra $\Walg$ is
the quantum Hamiltonian reduction $[U(\g)/U(\g)\{x-\langle\chi,x\rangle| x\in \mathfrak{m}\}]^{\operatorname{ad}\mathfrak{m}}$.

Let us list some important properties of the W-algebra.

1) The algebra $\Walg$ is naturally independent
of the choice of $\ell$ as was demonstrated in \cite{GG}. Moreover, it comes with a Hamiltonian
action of the group $Q=Z_G(e,h,f)$ by automorphisms.

2) Next, $\Walg$ comes with a filtration induced
from the  filtration on $U(\g)$, where $\deg \g(i)=i+2$. The associated graded for this filtration
is $\C[S]$, the algebra of functions on the Slodowy slice $S=e+\mathfrak{z}_\g(f)$.

3) Also note that the definition of $\Walg$ via the Hamiltonian reduction yields a homomorphism $U(\g)^G\rightarrow
\Walg$. As was checked by Ginzburg, see the footnote for \cite[Question 5.1]{Premet2},
this homomorphism is an isomorphism onto the center
of $\Walg$. So for $\lambda\in \h^*$ we can talk about the central reduction $\Walg_\lambda$.

Now let us discuss a reduction mod $p$ for W-algebras. Note that $\Walg$ is defined over some finite localization $R$
of $\Z$: we can take the Hamiltonian reduction $\Walg_R$ of $\U_R$ and  the properties 1), 2), 3) still hold.
So we can reduce mod $p$ and get the algebra $\Walg_\F:=\F\otimes_R \Walg_R$.

As Premet proved, see, for example, \cite[Theorem 2.1]{Premet4},
one has a central inclusion $\F[S^{(1)}]\hookrightarrow \Walg_\F$.
In \cite[Proposition 4.1]{Premet3}  Premet checked that one has an isomorphism $\U_\F^\chi\cong \operatorname{End}_\F(\mathfrak{m}_\F)\otimes \Walg_\F^\chi$.

\begin{Rem}
Consider $\U_{\F}|_{S^{(1)}}=\F[S^{(1)}]\otimes_{\F[\g^{(1)}]}\U_\F$.
One can strengthen Premet's result and show that  $\U_{\F}/\U_{\F}\{x-\langle\chi,x\rangle, x\in \mathfrak{m}_\F\}|_{S^{(1)}}$ is a Morita equivalence bimodule between $\U_{\F}|_{S^{(1)}}$ and
$\Walg_{\F}$. This follows from \cite{Topley}. From here we see that $\U_{\lambda,\F}/\U_{\lambda,\F}\{x-\langle\chi,x\rangle, x\in \mathfrak{m}_\F\}|_{S^{(1)}}$ is a Morita equivalence bimodule between $\U_{\lambda,\F}|_{S^{(1)}}$
and $\Walg_{\lambda,\F}$.
\end{Rem}

\subsection{Restriction functor for HC bimodules}\label{SS_HC_restr}
In this section we will recall results from \cite{HC} on the restriction functor
between the category of HC $\U$-bimodules and the category of HC $\Walg$-bimodules.

Namely in \cite{HC} the second author has constructed a functor $\bullet_{\dagger}:
\HC(\U)\rightarrow \HC^Q(\Walg)$ to the category of $Q$-equivariant HC $\Walg$-bimodules
(introduced in that paper) with the following properties:
\begin{enumerate}
\item The functor $\bullet_{\dagger}$ is exact, tensor, $\C[\h^*]^W$-bilinear and sends $\U$
to $\Walg$,
\item it maps $\HC_{\overline{\Orb}}(\U)$ to the category $\operatorname{Bimod}_{fin}^Q(\Walg)$
of finite dimensional $Q$-equivariant $\Walg$-bimodules,
\item and kills $\HC_{\partial\Orb}(\U)$.
\item There is a functor $\bullet^{\dagger}: \operatorname{Bimod}_{fin}^Q(\Walg)\rightarrow
\HC_{\overline{\Orb}}(\U)$ that is right adjoint to $\bullet_{\dagger}$.
\item For $\M\in \HC_{\overline{\Orb}}(\U)$, the kernel and the cokernel of the adjunction
unit $\M\rightarrow (\M_{\dagger})^{\dagger}$ are supported on $\partial \Orb$.
\item Let $\M\in \HC(\U)$ and $\mathcal{N}'\subset \M_{\dagger}$ be a $Q$-stable
subbimodule of finite codimension. Then there is a unique maximal subbimodule $\M'\subset \M$
with $\M'_{\dagger}=\mathcal{N}'$ and $\VA(\M/\M')=\overline{\Orb}$.
\end{enumerate}

We will need relative versions of (2)-(5), compare to \cite[Section 3.3.2]{cacti}.
Namely, let us pick an affine subspace $\h^1\subset\h^*$ and write
$\U_{\h^1}:=\C[\h^1]\otimes_{\C[\h^*]^W}\U,\Walg_{\h^1}:=\C[\h^1]\otimes_{\C[\h^*]^W}\Walg$.
Then we get a $\C[\h^1]$-bilinear exact tensor functor $\bullet_{\dagger}:\HC(\U_{\h^1})\rightarrow
\HC^Q(\Walg_{\h^1})$.

Let us write $\HC_{\overline{\Orb}}(\U_{\h^1})$ for the full subcategory of
$\HC(\U_{\h^1})$ consisting of HC bimodules $\mathcal{M}$ with $\VA(\M)\cap \mathcal{N}\subset\overline{\Orb}$.
The notation $\HC_{\partial \Orb}(\U_{\h^1})$ has the similar meaning. We also write
$\HC^Q_{\chi}(\Walg_{\h^1})$ for the category of all bimodules finitely generated
(as left, or equivalently, right) modules over $\C[\h^1]$.  Analogs of (2)-(5) are as follows.

\begin{itemize}
\item[(2$'$)] $\bullet_{\dagger}$ maps $\HC_{\overline{\Orb}}(\U_{\h^1})$ to
$\HC^Q_\chi(\Walg_{\h^1})$,
\item[(3$'$)] $\bullet_{\dagger}$ annihilates $\HC_{\partial\Orb}(\U_{\h^1})$.
\item[(4$'$)] There is a functor $\bullet^{\dagger}: \HC^Q_{\chi}(\Walg_{\h^1})
\rightarrow \HC_{\overline{\Orb}}(\U_{\h^1})$ right adjoint to $\bullet_{\dagger}$.
\item[(5$'$)] For $\M\in \HC_{\overline{\Orb}}(\U_{\h^1})$, the kernel and the cokernel
of the adjunction unit $\M\rightarrow (\M_{\dagger})^{\dagger}$ are supported
on $\HC_{\partial\Orb}(\U_{\h^1})$.
\end{itemize}

\subsection{Results on finite dimensional irreducible $\Walg$-modules}\label{SS_res_functor}
Let us state our results on the classification of finite dimensional irreducible $\Walg$-modules.
For this, we will need to recall one of the main results of \cite{HC}.  Since $Q$ acts on $\Walg$
by automorphisms, it also acts on the set
$\operatorname{Irr}_{fin}(\Walg)$ of the isomorphism classes of finite dimensional irreducible $\Walg$-modules.
Since the action of $Q$ on $\Walg$ is Hamiltonian, the action on $\operatorname{Irr}_{fin}(\Walg)$
descends to an action of component group $A(=A_{\Orb}):=Q/Q^\circ$.

One of the main results of \cite{HC}, see Section 1.2 there,
was a natural identification
$\operatorname{Irr}_{fin}(\Walg)/A\xrightarrow{\sim}\Prim_{\Orb}(\U)$: it sends $\J\in \Prim_\Orb(\U)$
to the $A$-orbit of the irreducible representations of $\Walg/\J_{\dagger}$, this is well-defined
and gives a bijection by (6) of Section \ref{SS_HC_restr}. So to finish the classification
of the finite dimensional irreducible $\Walg$-modules we need, for every primitive ideal $\J\in \operatorname{Prim}_\Orb(\U)$, to compute the stabilizer $H_\J$ (defined up to conjugacy) in
the $A$-orbit over $\J$.

Fix a central character $\lambda$ and assume for time being that it is regular.
Let us write $\Walg_{\lambda,\dcell}$ for the semisimple finite dimensional quotient
of $\Walg_\lambda$ whose simple representations are precisely the irreducibles
lying over the primitive ideals corresponding to the two-sided cell $\dcell$.
Recall the Springer representation $\Spr_{\Orb}:=H_{top}(\mathcal{B}_e,\C)$ of $W\times A$.
Also recall that $W_{[\lambda]}$ can be regarded as a subgroup of $W$ via $W_{[\lambda]}\hookrightarrow
W^a\twoheadrightarrow W$.
Let us write $\Spr_{\Orb,\dcell}$ for the sum of all irreducibles $W_{[\lambda]}$-submodules
in the Springer representation that belong to the family of irreducible $W$-modules indexed by the
cell $\dcell$.

\begin{Thm}\label{Thm:Walg_classif_regular}
Let $\dcell$ be a two-sided cell in $W_{[\lambda]}$ and let $\Orb=\Orb_{\dcell}$. Then the following is true.
\begin{enumerate}
\item Let $\J\in \Prim_{\Orb}(\U_\lambda)$ correspond to a left cell $\sigma\subset W_{[\lambda]}$
and let $H_\J$ denote a stabilizer in the $A$-orbit in $\Irr_{fin}(\Walg_\lambda)$ lying over $\J$.
Then the $A$-module $\Hom_{W_{[\lambda]}}([\sigma], \Spr_{\Orb})$ coincides with
the $A$-module induced from the trivial $H_{\J}$-module.
\item We have an isomorphism $K_0(\Walg_{\lambda,\dcell}\operatorname{-mod})\cong \Spr_{\Orb,\dcell}$
of $W_{[\lambda]}\times A$-modules.
\end{enumerate}
\end{Thm}

When $\lambda$ is integral, this theorem is the main result of \cite{LO}, see Theorem 1.1 and (iii)
of Theorem 7.4 there. Note that (1) is sufficient
to determine $H_\J$ (at least in all cases but 2). Indeed, the group $A$ is abelian for all nilpotent
orbits but twelve in the exceptional Lie algebras, see, e.g., \cite[Section 8.4]{CM}.
If $A$ is abelian, then $H_{\J}$ is just
the kernel of the $A$-action on   $\Hom_{W_{[\lambda]}}([\sigma], \Spr_{\Orb})$. Out of these twelve
cases, in ten cases we have $A=S_3$, where, clearly, the induced module determines a subgroup uniquely.
In the two remaining cases we have $A=S_4$ (in $F_4$) and $A=S_5$ (in $E_8$), we haven't checked
for general $\lambda$ if (1) determines $H_{\J}$ uniquely (though for an integral $\lambda$ this is indeed the case).

To finish this section let us explain what happens for singular central characters. The situation is very similar
to the integral case considered in \cite{LO}.
Let $\lambda_0$ be a singular  dominant element in $\h^*$.  Pick a dominant element
$\mu$ in the root lattice so that $\lambda:=\lambda_0+\mu$ is strictly dominant.
As was explained in Section \ref{SS_HC_prim},
$\Prim_{\Orb}(\U_{\lambda_0})\hookrightarrow \Prim_{\Orb}(\U_\lambda)$. This gives rise to
the partitions $\Prim_{\Orb}(\U_{\lambda_0})=\bigsqcup_{\dcell}\Prim_{\dcell}(\U_{\lambda_0}),
\Irr_{fin}(\Walg_{\lambda_0})=\bigsqcup_{\dcell}\Irr(\Walg_{\lambda_0,\dcell})$.

\begin{Cor}\label{Cor:singular}
Let $\J_0\subset \Prim_{\dcell}(\U_{\lambda_0})$, let $\J$ be the corresponding
ideal in $\Prim_{\dcell}(\U_\lambda)$. Then the $A$-orbit over $\J_0$ coincides
with $A/H_\J$ and $K_0(\Walg_{\lambda_0,\dcell}\operatorname{-mod})=\Spr_{\Orb,\dcell}^{W_{\lambda_0}}$.
\end{Cor}

%The proof repeats the integral case considered in \cite{LO}.

\subsection{Reduction of representations mod $p$}
Now fix a dominant  rational $\lambda\in \h^*$.
Recall, \cite[Theorem 1.3]{B_ineq}, that $\Walg_\lambda$ has a minimal ideal of finite codimension, say $\I$.
By definition, this ideal is defined over $\Q$. For a finite localization $R$ of $\Z$, set $\I_R:=\Walg_{\lambda,R}\cap \I$. We assume that $\gr\Walg_R=R[S]$ and $\gr\Walg_{\lambda,R}=R[S\cap \mathcal{N}]$,
this can be achieved after a finite localization of $R$.

\begin{Lem}\label{Lem:ideal_square}
After a finite localization of $R$, we get $\I_R^2=\I_R$.
\end{Lem}
\begin{proof}
Note that $\gr \Walg_{\lambda,R}/\I_R$ is a finitely generated commutative $R$-algebra. So after a
finite localization of $R$ we can achieve that $\Walg_{\lambda,R}/\I_R$ is a free finite rank $R$-module.
Note that $\Walg_{\lambda,R}$ is Noetherian because of $\gr\Walg_{\lambda,R}=R[S\cap \mathcal{N}]$.
In particular,  $\I_R$ is a finitely generated left $\Walg_{\lambda,R}$-module.
It follows that $\I_R/\I_R^2$ is a finitely generated module over $\Walg_{\lambda,R}/\I_R$ and hence
a finite rank $R$-module. Note that $\I_{\mathbb{Q}}$ is still the minimal ideal of finite
codimension in $\Walg_{\mathbb{Q}}$. So $\I_{\mathbb{Q}}=\I_{\mathbb{Q}}^2$. It follows
that $\I_R/\I_{R}^2$ is a finitely generated torsion $R$-module hence it is killed by
a finite localization of $R$.
\end{proof}

This lemma shows that $\left(\F\otimes_R (\Walg_{\lambda,R}/\I_{R})\right)\operatorname{-mod}$
is a Serre subcategory in $\Walg_{\lambda,\F}\operatorname{-mod}$.

After replacing $R$ with a finitely generated algebraic extension, we can assume that
$\Walg_{\lambda,\operatorname{Frac}(R)}/\I_{\operatorname{Frac}(R)}$ is split. So
there is a natural bijection  $\operatorname{Irr}_{fin}(\Walg_\lambda)\cong
\operatorname{Irr}(\Walg_{\lambda,\operatorname{Frac}(R)}/\I_{\operatorname{Frac}(R)})$.
So, for $L\in \operatorname{Irr}(\Walg_\lambda)$, we can talk about its reduction $L_\F$ mod $p$.
For standard reasons, $L_{\F}$ is irreducible.   As was checked in \cite[Section 6.5]{BL}, $L_{\F}$ has central
character $\chi$. So we get an inclusion $\operatorname{Irr}_{fin}(\Walg_\lambda)\hookrightarrow
\operatorname{Irr}(\Walg^\chi_{\lambda,\F})$. Recall, Section \ref{SS_W_background},
that the target is naturally identified with $\operatorname{Irr}(\U_{\lambda,\F}^\chi)$.

\begin{Prop}\label{Prop:simples_from_char_0}
For $p\gg 0$, the image of $\operatorname{Irr}_{fin}(\Walg_\lambda)$
in $\operatorname{Irr}(\U_{\lambda,\F}^\chi)$ consists of the simples with degree
of dimension polynomial equal to $\dim \Orb/2$.
\end{Prop}
\begin{proof}
The simples in $\operatorname{Irr}(\U_{\lambda,\F}^\chi)$ with degree of dimension polynomial
equal to $\dim \Orb/2$ correspond to the simples in $\operatorname{Irr}(\Walg^\chi_{\lambda,\F})$
whose dimension is independent of $p$. Let $d$ be the maximal dimension of these representations.
Let $\I'_R$ be the ideal in $\Walg_R$ generated by the elements $\sum_{\sigma\in S_{2d}}\operatorname{sgn}(\sigma)a_{\sigma(1)}\ldots a_{\sigma(2d)}$ for $a_i\in \Walg_R, i=1,\ldots,2d$.
By the Amitsur-Levitsky theorem, the ideal $\I'_R$ vanishes on all representations of $\Walg_\F$ of
dimension $\leqslant d$. Arguing as in \cite[Lemma 5.1]{rouq_der}, we see that $\I'_{\C}$ is of finite
codimension. So $\I'_\C\supset \I$. It follows that after a finite localization of $R$, we have
$\I'_R\supset \I_R$. So any irreducible representation of $\Walg_{\lambda,\F}$ of dimension $\leqslant d$
factors through $\Walg_{\lambda,\F}/\I_{\F}$ for $p$ large enough. This finishes the proof.
\end{proof}

\subsection{Proof of Theorem \ref{Thm:Walg_classif_regular}}
Let us prove Theorem \ref{Thm:Walg_classif_regular} in the case when $\lambda$ is rational.

\begin{proof}[Proof of Theorem \ref{Thm:Walg_classif_regular} for rational $\lambda$]
Let us start by proving (2). From Proposition \ref{Prop:simples_from_char_0} combined
with (2) of Theorem \ref{Thm:main} we know that $K_0(\Walg_{\lambda,\dcell}\operatorname{-mod})=H_*(\mathcal{B}_e,\C)_{\dcell}$,
the sum of all irreducible $W_{[\lambda]}$-submodules in $H_*(\mathcal{B}_e,\C)_{\dcell}$ that belong to
$\dcell$. This is with respect to the standard embedding $W_{[\lambda]}\hookrightarrow W^a$.
What remains to show is that
\begin{equation}\label{eq:homol_isot_equal} H_*(\mathcal{B}_e,\C)_{\dcell}\xrightarrow{\sim} H_{top}(\mathcal{B}_e,\C)_{\dcell}
\end{equation}
where now the action on the right hand side is via $W_{[\lambda]}\hookrightarrow W^a\twoheadrightarrow W$
and the map is induced by the natural projection $H_*(\mathcal{B}_e,\C)\twoheadrightarrow H_{top}(\mathcal{B}_e,\C)$.
According to Dodd, \cite[Section 7]{Dodd}, $K_0(\Walg_{\lambda,\dcell}\operatorname{-mod})\subset H_*(\mathcal{B}_e,\C)$
projects injectively to $H_{top}(\mathcal{B}_e,\C)$. The projection $H_*(\mathcal{B}_e,\C)\rightarrow H_{top}(\mathcal{B}_e,\C)$ is $W^a$-equivariant, where on the target space $W^a$ acts via the projection
$W^a\twoheadrightarrow W$, and so intertwines the actions of $W_{[\lambda]}$.
This implies (\ref{eq:homol_isot_equal}) and finishes the proof of (2).

Let us now deduce (1) from (2). The  restriction functor $\bullet_{\dagger}:\HC_{\Orb}(\U_\lambda)\rightarrow \operatorname{Bimod}_{fin}^Q(\Walg_\lambda)$ recalled in Section \ref{SS_res_functor}
equips $K_0(\Walg_{\lambda,\dcell}\operatorname{-mod})$ with an action of $K_0(\HC^{ss}_{\dcell}(\U_\lambda))=\Jalg_{\dcell}(W_{[\lambda]})$.
By results of Dodd, \cite[Section 8]{Dodd}, this action is the same
as the $\Jalg_{\dcell}(W_{[\lambda]})$-action on $K_0(\U_{\lambda,\F}^\chi\operatorname{-mod})_{\dcell}$.

By the description of the $\Jalg_{\dcell}$-action in the previous paragraph, the span of classes of the
irreducible modules lying over $\J$ equals $t_d K_0(\Walg_{\lambda,\dcell}\operatorname{-mod})$.

So $t_d K_0(\Walg_{\lambda,\dcell}\operatorname{-mod})$  is nothing else
but $\Hom_{W_{[\lambda]}}([\sigma], \Spr_{\Orb})$. On the other hand, the span is $A$-stable and is
the $A$-representation induced from the trivial representation of $H_\J$. This finishes the proof of (1).
\end{proof}

Now let us reduce the proof of Theorem \ref{Thm:main} to the case when $\lambda$ is rational.
To state our main technical result we need some notation.

Pick a regular central character $\lambda\in \h^*$.  Let $W_0$ denote a minimal parabolic subgroup
of $W$ containing $W_{[\lambda]}$, we can conjugate $\lambda$ and  assume that $W_0$ is standard,
while $\lambda$ is still dominant.
We can write $\lambda$ as $\lambda_1+\lambda_2$, where $\lambda_1$ lies in $(\h^*)^{W_0}$ and
$\lambda_2$ lies in the orthogonal complement to $(\h^*)^{W_0}$. Note that $\lambda_2$ is rational.

\begin{Prop}\label{Prop:rat_reduction}
For all dominant regular rational $\lambda'\in \lambda_2+(\h^*)^{W_0}$
satisfying $W_{[\lambda']}=W_{[\lambda]}$ the following is true. Let $\sigma$
be a left cell in $W_{[\lambda]}$ and let $\J,\J'$ be primitive ideals in $\U$ with central characters
$\lambda,\lambda'$ corresponding  to the left cell $\sigma$. Then $H_{\J}=H_{\J'}$.
\end{Prop}
\begin{proof}
The proof is in several steps.

{\it Step 1}.
Since $\lambda$ is not rational, $W_0\neq W$.  By \cite[Proposition 5.7]{LO} for any integral dominant
$\mu\in (\h^*)^{W_0}$ for the ideal $\J^\mu$ with central character $\lambda+\mu$
(corresponding to $\J$ under the isomorphism $\operatorname{Prim}(\U_\lambda)
\xrightarrow{\sim} \operatorname{Prim}(\U_{\lambda+\mu})$) we have $H_{\J^\mu}=H_{\J}$.
So in the proof we can assume  that $\lambda_1$ is Zariski generic in $(\h^*)^{W_0}$.
To simplify the notation we will write $\h^1$ for $\lambda_2+(\h^*)^{W_0}$.

{\it Step 2}.
A standard argument, see, for example, the proof of \cite[Lemma 5.1]{rouq_der}, shows that there is an
ideal $\I\subset \Walg_{\h^1}$ such that $\Walg_{\h^1}/\I$ is finitely generated over
$\C[\h^1]$ and for a Weil generic $\hat{\lambda}\in \h^1$ the specialization
$\I_{\hat{\lambda}}$ is the minimal ideal of finite codimension in $\Walg_{\hat{\lambda}}$.
Let us write $\bar{\J}$ for the kernel of $\U_{\h^1}\rightarrow (\Walg_{\h^1}/\I)^{\dagger}$.
Note that, thanks to (4$'$) of Section \ref{SS_HC_restr}, $\VA(\U_{\h^1}/\bar{\J})\cap \mathcal{N}=\overline{\Orb}$. So $\Walg_{\h^1}/\bar{\J}_{\dagger}$
is finitely generated over $\C[\h^1]$. Since $\bar{\J}_{\dagger}\subset \bar{\I}$,
we can replace $\I$ with $\bar{\J}_{\dagger}$ and assume that $\I=\bar{\J}_{\dagger}$.

{\it Step 3}.
For a dominant regular $\hat{\lambda}$ with $W_{[\hat{\lambda}]}=W_{[\lambda]}$,
let us write $\J_{\hat{\lambda}}$ for the primitive ideal with central character $\hat{\lambda}$
corresponding to the left cell $\sigma$. Let us prove that for a Zariski generic $\hat{\lambda}\in \h^1$
with   $W_{[\hat{\lambda}]}=W_{[\lambda]}$ we have $\bar{\J}_{\hat{\lambda}}\subset \J_{\hat{\lambda}}$.
Note that this inclusion is automatic provided $\I_{\hat{\lambda}}$ is the minimal ideal of finite
codimension in $\Walg_{\hat{\lambda}}$.

Pick $w\in \sigma$. For $\hat{\lambda}\in \h^1$
consider the Verma module $\Delta(w\hat{\lambda})$. These Verma modules form a flat family
over $\h^1$, let us denote the corresponding $\U_{\h^1}$-module
by $\Delta_{\h^1}$. Inside we have a $\U_{\h^1}$-submodule $\bar{\J}\Delta_{\h^1}$.
Consider the quotient $\Delta_{\h^1}/\bar{\J}\Delta_{\h^1}$. Its specialization to
$\hat{\lambda}$ is $\Delta(w\hat{\lambda})/\bar{\J}_{\hat{\lambda}}\Delta(w\hat{\lambda})$.
It is nonzero for a Weil generic $\hat{\lambda}$. It follows that it is nonzero for
a Zariski generic $\hat{\lambda}$ as well. The claim in the first paragraph of this step
follows.

{\it Step 4}. The algebra $\Walg_{\h^1}/\I$ comes with a $Q$-action. Since the algebra
$\Walg_{\h^1}/\I$ is a  finitely generated module over $\C[\h^1]$, we see that the number of irreducible
representations is the same for two Zariski generic specializations (this is a version of the
Tits deformation argument). Moreover, for two nearby generic parameters, there is a natural bijection
between the irreducibles. Being natural, this bijection preserves the stabilizers in
$Q$. And when the parameters are not nearby, the monodromy may appear but it does
not change the stabilizers in $A$.

{\it Step 5}. Recall, Step 1, that we can assume that $\lambda$ is Zariski generic in $\h^1$.
Similarly, we can assume that $\lambda'$ is Zariski generic.  Now the claim of the proposition
follows from Step 4.
\end{proof}

\begin{proof}[Proof of Theorem \ref{Thm:Walg_classif_regular} for general $\lambda$]
(1) for $\lambda$ immediately follows from Proposition \ref{Prop:rat_reduction} and (1)
for $\lambda'$ proved above. To prove (2) we can argue as follows. Take a Weil generic
$\hat{\lambda}$ with $W_{[\hat{\lambda}]}=W_{[\lambda]}$. Then we  have the degeneration
maps (compare to \cite[Section 11.1]{BL}) $$K_0(\Walg_{\hat{\lambda}}\operatorname{-mod}_{fin})\rightarrow
K_0(\Walg_{\lambda'}\operatorname{-mod}_{fin}), K_0(\Walg_{\hat{\lambda}}\operatorname{-mod}_{fin})
\rightarrow K_0(\Walg_{\lambda}\operatorname{-mod}_{fin}).$$ By the proof of Proposition
\ref{Prop:rat_reduction}, we see that both these maps are isomorphisms. They are also
$W_{[\lambda]}$-invariant. This implies (2).
\end{proof}

\subsection{Proof of Corollary \ref{Cor:singular}}
Let $\lambda_0,\lambda$ be as before the statement of Corollary \ref{Cor:singular}.
Let $V$ be a finite dimensional $G$-module. We can define the endo-functor
$V\otimes\bullet$ of $\Walg\operatorname{-mod}$ as the tensor product with
the bimodule $(V\otimes \U)_{\dagger}$. Using this we can define translation
functors $\mathcal{T}_{\lambda_0\rightarrow \lambda}:
\Walg\operatorname{-mod}_{\lambda_0}\rightarrow \Walg\operatorname{-mod}_{\lambda},
\mathcal{T}_{\lambda\rightarrow \lambda_0}: \Walg\operatorname{-mod}_{\lambda}
\rightarrow \Walg\operatorname{-mod}_{\lambda_0}$ in a standard way. They enjoy properties similar
to those of the usual translation functors (because $\bullet_{\dagger}$ is a tensor functor):
\begin{enumerate}
\item On $K_0(\Walg_{\lambda_0}\operatorname{-mod}_{fin})$ the composition $\mathcal{T}_{\lambda\rightarrow \lambda_0}\circ \mathcal{T}_{\lambda_0\rightarrow \lambda}$ is the the multiplication by
    $|W_{\lambda_0}|$.
\item On $K_0(\Walg_{\lambda}\operatorname{-mod}_{fin})$ the composition
$\mathcal{T}_{\lambda_0\rightarrow \lambda}\circ\mathcal{T}_{\lambda\rightarrow \lambda_0}$
acts as $\sum_{w\in W_{\lambda_0}}w$.
\end{enumerate}
This implies the equality $K_0(\Walg_{\lambda_0}\operatorname{-mod}_{fin})=K_0(\Walg_{\lambda}\operatorname{-mod}_{fin})^{W_{\lambda_0}}$.
Also, for $w$ longest in its right $W_{\lambda_0}$-coset, the maps $[\mathcal{T}_{\lambda_0\rightarrow \lambda}],
[\mathcal{T}_{\lambda\rightarrow \lambda_0}]$ map between $K_0(\Walg/\J(w \lambda_0)_{\dagger}\operatorname{-mod})$
and $K_0(\Walg/\J(w(\lambda))_\dagger\operatorname{-mod})$, which together with (1) implies that
$H_{\J(w\lambda_0)}=H_{\J(w(\lambda))}$.

\section{Application to real variation of stability conditions}
\label{izvrat}
In this section we use Theorem \ref{Thm:main} to essentially realize the idea sketched in \cite[Remark 6]{ABM}.

We now describe a simplified version of a construction of \cite{ABM}. Let $\A$ be an abelian category
and let $\zeta:\Ce \to K_0(\A)^*$ be a polynomial map. We assume that for some $r>0$ we have
\begin{equation}\label{pos}
\langle [M], \zeta(x)\rangle \in \RE_{>0}\ \ \ \ \forall \ x\in (0,r), \, M\in \A,\, M\ne 0.
\end{equation} In this situation
we get a filtration on $\A$ by Serre subcategories where $\A_{\geqslant d}$ consists of objects
$M$ such that the polynomial $x \mapsto \langle [M], \zeta(x)\rangle$ has a zero of order
at least $d$ at zero. We say that a derived equivalence
$\phi:D^b(\A)\to D^b(\A')$ is a {\em perverse equivalence governed by $\zeta$} if there exists
a filtration on $\A'$ by Serre subcategories such that $D^b(\A'_{\geqslant d})= \phi( D^b(\A_{\geqslant d}))$,
while the functor $gr_d(\phi): D^b(\A_{\geqslant d}/\A_{> d})\to D^b(\A'_{\geqslant d}/\A'_{> d})$
sends $\A_{\geqslant d}/\A_{> d}$ to $(\A_{\geqslant d}/\A_{> d})[d]$.

We now set $\A=\U_{\lambda,\F}\operatorname{-mod}^{\chi}$.
Let $\xi:\RE\to \t ^*_\RE$ be an affine linear functional sending zero to a face $F$ in the closure
of the fundamental alcove $A_0$; we assume that $\xi(\RE_{>0})$ intersects $A_0$.

 The central charge map $Z:\t \to K_0(\operatorname{Coh}(\B_e))^* $ was defined in \cite{ABM}.
We use identification \eqref{identi} for a choice of $\lambda$
in the fundamental $p$-alcove to get a map $\t\to K_0(\A)^*$ which we
also denote by $Z$. We set $\zeta=Z\circ \xi$. Then \cite[Proposition 1(a)]{ABM}
implies that the positivity condition \eqref{pos} holds for some $r>0$.

The face $F$ determines a proper subset in the set of vertices of the affine Dynkin graph, let
$W_F$ be the corresponding finite Weyl group and
$w_F$ be the longest
element in $W_F$; let $\tilde w_F$ be the canonical (minimal length) lift
of $w_F$ to the affine braid group $B_{aff}$. Note that a path
in the complement to affine coroot hyperplanes whose end-points are contained in $t^*_{\RE}$
defines an element in $B_{aff}$; the element $\tilde w_F$ corresponds to the path $[0,1]\to \t^*_\Ce$,
$x\mapsto \xi(R\exp(2\pi i x))$ for a small $R>0$.

Recall the action of $B_{aff}$ on $D^b(\A)$.

The main result of this section is as follows.

\begin{Thm}
The functor $\tilde w_F:D^b(\A)\to D^b(\A)$ is a perverse equivalence governed by
$\zeta$.
\end{Thm}

\begin{proof}
We let $\A_{\geqslant d}$ be the filtration introduced earlier in this section and we let
$\A'_{\geqslant d}$ denote the Serre subcategory generated by $\U_{\lambda,\F}\operatorname{-mod}^\chi_{\leqslant \dcell}$
for all cells $\dcell$ in $W_F$ such that $a(\dcell)\geqslant d$.

The result clearly follows from the following two statements.

a) We have $\A_{\geqslant d}=\A'_{\geqslant d}$.

b)  The functor $\tilde w_F[-d]$ induces a $t$-exact functor on $D^b(\A'_{\geqslant d}/\A'_{> d})$.

We claim that (a)  follows from Theorem \ref{Thm:main}.
To see this, choose a regular rational weight $\lambda$ with $W_{[\lambda]}=W_F$. Furthermore, we can
and will assume that $\lambda$ satisfies the following assumptions: it can be written as
$\lambda=\mu+\nu$, where $\nu$ is an integral weight and
$\mu$ lies in the closure of the fundamental alcove $A_0$, while $\mu + t\nu$ lies in
$A_0$ for small (equivalently, for some) $t>0$ (equivalently, $\mu'+ t\nu$ lies in $A_0$
for all $\mu'\in F$ and small $t>0$, where the bound on $t$ depends on $\mu'$).
Choose a large prime $p$ such that $(p+1) \lambda$ is an integral weight.
Then $\tilde \lambda:=(p+1)\mu +\nu$ is an integral weight satisfying:
$\tilde \lambda = \lambda \mod p$ and $\frac{\tilde \lambda}{p}\in A_0$.

For $M\in \A$ consider the %dimension
 polynomial $D_M$, such that
for an integral weight $\eta$ such that $\frac{\tilde{\lambda}+\eta -\rho}{p}\in A_0$
we have $\dim (T_{\tilde{\lambda}\to \eta}(M))=D_M(\lambda)$, where $T$ denotes the translation
functor (the existence of $D_M$ follows from \cite[Theorem 6.2.1]{BMR}). By Theorem \ref{Thm:main} for a $M\in \A'_{\geqslant d}$, $M\not \in   \A'_{> d}$
we have $\deg_p(D_M)= \dim (\B)-d$.

 On the other hand
the central charge $Z$ of \cite{ABM}
satisfies
%is related to the dimension polynomials by the formula
(see the proof of  \cite[Proposition 1]{ABM}):
$$\langle Z(\frac{\eta+\rho}{p}), [M] \rangle = p^{-\dim(\B)}D_M(\eta).$$
It follows that the order of zero of the polynomial $\zeta_M(t)=\langle Z(\mu+t\nu), [M] \rangle$
at $t=0$ equals  $\dim (\B)-\deg_p(D_M)$, this proves (a).

We now sketch the proof of (b). We use the fact that the braid group
$B_F$ of the Coxeter group $W_F$ acts  the category
$D^b(\U_{\lambda,R}\operatorname{-mod})$ compatibly with its
action on $D^b(\U_{\lambda,\F }\operatorname{-mod})$.

Note that the action of $\tilde{w}_F$ is given by the derived tensor product with the
wall-crossing $\U_\lambda$-bimodule $\mathsf{WC}_{w_F}$ corresponding to the element
$w_F$. By \cite[Theorem 3.1]{perv}, the functor $\mathcal{WC}_{w_F}$ is a perverse
equivalence with
$$\U_{\lambda}\operatorname{-mod}_{\geqslant d}=\{M\in \U_\lambda\operatorname{-mod}|
\dim \VA(\U_\lambda/\operatorname{Ann}(M))\leqslant \dim \mathcal{N}-2d\}.$$
As was shown in the proof of that theorem the statement reduces to vanishing of
Tor's involving $\mathsf{WC}_{w_F}$ and the quotients of $\U_\lambda$ by the minimal
ideals with given dimensions of associated varieties. This vanishing was checked in
the proof.  Now this vanishing over $\C$ implies the analogous vanishing over $R$
(after a finite localization) and hence the claim that the endo-functor
$\mathcal{WC}_{w_F}=\mathsf{WC}_{w_F,R}\otimes^L_{\U_{\lambda,R}}\bullet$
of $D^b(\U_{\lambda,\F}\operatorname{-mod}^\chi)$ is perverse with respect to
the filtration $\A'_{\geqslant d}$.
%By \cite[Proposition 4.1]{BFO}, the action of $\tilde w_F$ on the derived category of  Harish-Chandra
%bimodules is a perverse equivalence relative to the filtration by the dimension of support.
%The action of $\tilde{w}_F$ can be expressed as the derived tensor product with a Harish
%Notice also that the action of $\tilde w_F$ can be expressed as the derived tensor product
%with a complex of bimodules (in fact, with a single bimodule, by this fact is irrelevant
%for the present argument). The claim now follows from transitivity of derived tensor product as in the proof
%of Proposition  \ref{Prop:categor_filtration}.
\end{proof}


\begin{thebibliography}{99}
\bibitem[B]{B} R. Bezrukavnikov,  {\em On two geometric realizations of an affine Hecke algebra,}
    Publ. IHES 123(1) (2016), 1--67.
%\bibitem[BFG]{BFG} R. Bezrukavnikov, M. Finkelberg, V. Ginzburg,
%{\it Cherednik algebras and Hilbert schemes in characteristic
%p}. With an appendix by Pavel Etingof.
%Represent. Theory 10 (2006), 254-298.
\bibitem[ABM]{ABM} R. Anno, R. Bezrukavnikov, I. Mirkovic,
{\em Stability conditions for Slodowy slices and real variations of stability,}  Mosc. Math. J. 15 (2015), no. 2, 187-203, 403.
\bibitem[BFO]{BFO} R. Bezrukavnikov, M. Finkelberg, V. Ostrik,
{\it Character D-modules via Drinfeld center of Harish-Chandra bimodules}. Invent. Math. 188 (2012), no. 3, 589-620.
\bibitem[BL]{BL} R. Bezrukavnikov, I. Losev, {\it Etingof conjecture for quantized quiver varieties}.
http://www.northeastern.edu/iloseu/bezpaper.pdf
\bibitem[BM]{BM} R. Bezrukavnikov, I. Mirkovic. {\it Representations of semisimple Lie algebras in prime characteristic
and noncommutative Springer resolution}.  Ann. Math. 178 (2013), n.3, 835-919.
\bibitem[BMR1]{BMR_sing} R. Bezrukavnikov, I. Mirkovic, D. Rumynin.
{\it Singular localization and intertwining functors for reductive Lie algebras in prime characteristic}.
Nagoya Math. J.  184  (2006), 1–55.
\bibitem[BMR2]{BMR} R. Bezrukavnikov, I. Mirkovic, D. Rumynin. {\it Localization of modules for a semisimple Lie algebra
in prime characteristic} (with an appendix by R. Bezrukavnikov and S. Riche), Ann. of Math. (2)
167 (2008), no. 3, 945-991.
\bibitem[BR]{BR} R. Bezrukavnikov, S. Riche,
{\it Affine braid group actions on derived categories of Springer resolutions}.
Ann. Sci. \'{E}c. Norm. Sup\'{e}r. (4) 45 (2012), no. 4, 535-599 (2013).
\bibitem[BY]{BY} R. Bezrukavnikov, Z. Yun, {\em On Koszul duality for Kac-Moody groups,}
Represent. Theory {\bf 17} (2013), 1-98.
\bibitem[CM]{CM} D. Collingwood, W. McGovern. {\it Nilpotent orbits in semisimple Lie algebras}.
Chapman and Hall, London, 1993.
\bibitem[D]{Dodd} C. Dodd. {\it Injectivity of a certain cycle map for finite dimensional W-algebras}.
 Int. Math. Res. Not.  2014, no. 19, 5398–5436.
\bibitem[GG]{GG} W.L. Gan, V. Ginzburg. {\it Quantization of Slodowy slices}. IMRN, 5(2002), 243-255.
\bibitem[Ja]{Jantzen} J.C. Jantzen. {\it Einh\"{u}llende Algebren
halbeinfacher Lie-Algebren}. Ergebnisse der Math., Vol. 3, Springer,
New York, Tokio etc., 1983.
\bibitem[Jo1]{Joseph_variety} A. Joseph. {\it On the associated variety of a primitive ideal}.
J. Algebra, 93 (1985), no. 2, 509-523.
\bibitem[Jo2]{Joseph_cyclicity} A. Joseph. {\it On the cyclicity of vectors associated
with Duflo involutions}. in Lecture Notes in Mathematics, vol. 1243, 144-188, Springer-Verlag,
Berlin, 1987.
\bibitem[KW]{KW} V.G. Kac, B.Yu. Weisfeiler, {\it The irreducible representations of Lie
$p$-algebras.} Funkcional. Anal. i Prilozen. 5 1971 no. 2, 28-36 (in Russian).
\bibitem[L1]{W_ICM} I. Losev, {\it Finite W-algebras}.  Proceedings of the International Congress of Mathematicians
Hyderabad, India, 2010, p. 1281-1307.
\bibitem[L2]{HC} I.V. Losev. {\it Finite dimensional representations of
W-algebras}. Duke Math J. 159(2011), n.1, 99-143.
%\bibitem[L3]{W_dim} I. Losev, {\it Dimensions of irreducible modules over W-algebras and Goldie ranks}.
%Invent. Math. 200 (2015), N3, 849-923.
\bibitem[L3]{rouq_der} I. Losev. {\it Derived equivalences for Rational Cherednik algebras}.
Duke Math J. 166(2017), N1, 27-73.
\bibitem[L4]{B_ineq} I. Losev. {\it Bernstein inequality and holonomic modules} (with a joint appendix
by I. Losev and P. Etingof).   Adv. Math. 308 (2017),  941-963.
\bibitem[L5]{cacti} I. Losev. {\it Cacti and cells}. arXiv:1506.04400. Accepted by J. Eur. Math. Soc.
\bibitem[L6]{perv} I. Losev. {\it Wall-crossing functors for quantized symplectic resolutions: perversity and partial Ringel dualities}. arXiv:1604.06678.
\bibitem[LO]{LO} I. Losev, V. Ostrik, {\it Classification of finite dimensional irreducible modules
over W-algebras}. Compos. Math. 150(2014), N6, 1024-1076.
\bibitem[Lu1]{orange}  G. Lusztig. {\it Characters of reductive groups over a finite field}, Ann. Math. Studies 107, Princeton University Press (1984).
\bibitem[Lu2]{equiv_K} G. Lusztig. {\it Equivariant
K-theory and representations of Hecke algebras.}
Proc. Amer. Math. Soc. 94 (1985), no. 2, 337-342.
\bibitem[Lu3]{Lusztig_leading} G. Lusztig.  {\it Leading coefficients of character values of Hecke algebras}, Proc.Symp.Pure Math.47(2), Amer.Math.Soc. 1987, 235-262.
\bibitem[Lu4]{cells4} G. Lusztig. {\it Cells in affine Weyl groups, IV}.
J. Fac. Sci. Univ. Tokyo Sect. IA, Math. 36 (1989), 297-328.
\bibitem[M]{Milicic} D. Milicic.  {\it Localization and Representation Theory of Reductive Lie Groups}, available at
http://www.math.utah.edu/$\sim$milicic.
%\bibitem[P1]{Premet_KW} A. Premet, \emph{Irreducible representations of Lie algebras of reductive groups and the %Kac-Weisfeiler conjecture}, Invent. Math. 121(1995), 79-117.
\bibitem[P1]{Premet1} A. Premet. {\it Special transverse slices and their enveloping algebras}. Adv. Math. 170(2002),
1-55.
\bibitem[P2]{Premet2} A. Premet. {\it Enveloping algebras of Slodowy slices and the Joseph ideal}.
J. Eur. Math. Soc, 9(2007), N3, 487-543.
\bibitem[P3]{Premet3} A. Premet.  {\it Primitive ideals, non-restricted representations and finite W-algebras}.
Mosc. Math. J., 7 (2007) N 4,  743–762.
\bibitem[P4]{Premet4} A. Premet, {\it Commutative quotients of finite W-algebras}. Adv. Math. 225 (2010),  N1,  269-306.
\bibitem[R]{Riche} S. Riche, {\it Geometric braid group action on derived categories of coherent sheaves}. Represent. Theory 12 (2008) 131-169. With appendix by R. Bezrukavnikov and S. Riche.
\bibitem[T]{Topley} L. Topley, {\it A Morita theorem for modular finite
W-algebras}.  Math. Z. 285 (2017), N 3-4, 685-705.
\end{thebibliography}
\end{document}